\documentclass[final]{siamart}
\usepackage[T1]{fontenc}
\usepackage[latin9]{inputenc}
\usepackage{amstext}
\usepackage{amssymb}
\usepackage{graphicx}

\makeatletter

\providecommand{\tabularnewline}{\\}

\numberwithin{equation}{section}
\newsiamthm{thm}{\protect\theoremname}
  \newsiamremark{assume}{Assumption}
\newsiamremark{rem}{\protect\remarkname}
  \newsiamremark{algor}{Algorithm}
\newsiamthm{lem}{\protect\lemmaname}
\newsiamthm{cor}{\protect\corollaryname}
\newsiamthm{prop}{\protect\propositionname}

\usepackage{etoolbox}
\usepackage{hyperref}
\usepackage{color}

\newcommand{\TheTitle}{GMRES-Accelerated ADMM for Quadratic Objectives}  
\newcommand{\TheAuthors}{R. Y. Zhang and J. K. White}
\headers{\TheTitle}{\TheAuthors}

\usepackage[font={footnotesize,it},labelfont=sc]{caption}

\BeforeBeginEnvironment{tabular}{\scriptsize}

\@ifundefined{showcaptionsetup}{}{%
 \PassOptionsToPackage{caption=false}{subfig}}
\usepackage{subfig}
\makeatother

  \providecommand{\corollaryname}{Corollary}
  \providecommand{\lemmaname}{Lemma}
  \providecommand{\propositionname}{Proposition}
  \providecommand{\remarkname}{Remark}
\providecommand{\theoremname}{Theorem}

\begin{document}

\title{GMRES-Accelerated ADMM for Quadratic Objectives\thanks{This work was supported in part by the Skolkovo-MIT initiative in
Computational Mathematics.}}

\author{Richard Y. Zhang\thanks{Department of Electrical Engineering and Computer Science, Massachusetts
Institute of Technology, Cambridge, MA 02139 (\texttt{ryz@alum.mit.edu}).} \and Jacob K.White\thanks{Department of Electrical Engineering and Computer Science, Massachusetts
Institute of Technology, Cambridge, MA 02139 (\texttt{white@mit.edu}). }}
\maketitle
\begin{abstract}
We consider the sequence acceleration problem for the alternating
direction method-of-multipliers (ADMM) applied to a class of equality-constrained
problems with strongly convex quadratic objectives, which frequently
arise as the Newton subproblem of interior-point methods. Within this
context, the ADMM update equations are linear, the iterates are confined
within a Krylov subspace, and the General Minimum RESidual (GMRES)
algorithm is optimal in its ability to accelerate convergence. The
basic ADMM method solves a $\kappa$-conditioned problem in $O(\sqrt{\kappa})$
iterations. We give theoretical justification and numerical evidence
that the GMRES-accelerated variant consistently solves the same problem
in $O(\kappa^{1/4})$ iterations for an order-of-magnitude reduction
in iterations, despite a worst-case bound of $O(\sqrt{\kappa})$ iterations.
The method is shown to be competitive against standard preconditioned
Krylov subspace methods for saddle-point problems. The method is embedded
within SeDuMi, a popular open-source solver for conic optimization
written in MATLAB, and used to solve many large-scale semidefinite
programs with error that decreases like $O(1/k^{2})$, instead of
$O(1/k)$, where $k$ is the iteration index.
\end{abstract}



\section{Introduction}

\global\long\def\Lag{\mathscr{L}}
\global\long\def\R{\mathbb{R}}
\global\long\def\C{\mathbb{C}}
\global\long\def\S{\mathbb{S}}
\global\long\def\P{\mathbb{P}}
\global\long\def\diag{\mathrm{diag}\,}
\global\long\def\GMRES{\mathbf{GMRES}}
\global\long\def\lb{\mathrm{lb}}
\global\long\def\Real{\mathrm{Re}\,}
\global\long\def\Imag{\mathrm{Im}\,}
\global\long\def\tr{\mathrm{tr}\,}
The alternating direction method-of-multipliers (ADMM)~\cite{glowinski1975approximation,gabay1976dual}
is a popular first-order optimization algorithm, used to establish
consensus between many local subproblems and a global master problem.
This sort of \emph{decomposable} problem structure naturally arises
over a wide range of applications, from statistics and machine learning,
to the solution of semidefinite programs; see~\cite{boyd2011distributed}
for an overview. Part of the appeal of ADMM is that it is simple and
easy to implement at a large scale, and that convergence is guaranteed
under very mild assumptions. On the other hand, ADMM can converge
very slowly, sometimes requiring thousands of iterations to compute
solutions accurate to just 2-3 digits.

Sequence acceleration\textemdash the idea of extrapolating information
collected in past iterates\textemdash can improve the convergence
rates of first-order methods, thereby allowing them to compute more
accurate solutions. Within this context, the technique known as momentum
acceleration has been particularly successful. Originally developed
by Nesterov for the basic gradient method~\cite{nesterov1983method},
the technique has since been extended to a wide range of first-order
methods~\cite{nesterov2005smooth,beck2009fast,becker2011nesta}.
In each case, the momentum-accelerated variant improves upon the basic
first-order method by an order of magnitude, achieving a convergence
rate that is provably optimal. 

Unfortunately, sequence acceleration has been less successful for
ADMM. Momentum-accelerated ADMM schemes have been proposed~\cite{goldstein2014fast,patrinos2014douglas,ouyang2015accelerated,kadkhodaie2015accelerated},
but the reductions in iterations have been modest, and limited to
special cases where the objectives are assumed to be strongly convex
and/or quadratic. Part of the problem is that ADMM is a difficult
method to analyze, even in the simple quadratic case. The most widely
used technique to accelerate ADMM in practice is simple over-relaxation,
which reliably reduces total iteration count by a small constant~\cite{wen2010alternating,giselsson2014diagonal,ghadimi2015optimal,nishihara2015general}.
It remains unclear whether an order-of-magnitude acceleration is even
achievable for ADMM in the first place. 

\subsection{\label{subsec:intro_gmres}Sequence acceleration using GMRES}

In this paper, we restrict our attention to the ADMM solution of the
equality-constrained quadratic program
\begin{alignat}{2}
 & \text{minimize } & \frac{1}{2}x^{T}Dx+c^{T}x+p^{T}z & \tag{ECQP}\label{eq:ecqp}\\
 & \text{subject to } & Ax+Bz=d,\nonumber 
\end{alignat}
which frequently arises as the Newton subproblem of interior-point
methods. The Karush\textendash Kuhn\textendash Tucker (KKT) equations
for (\ref{eq:ecqp}) are linear, so first-order methods applied to
(\ref{eq:ecqp}) reduce to a matrix iteration of the form
\begin{equation}
u^{k+1}=G(\beta)u^{k}+b(\beta),\label{eq:admm_iters_intro}
\end{equation}
where $u=[x;z;y]$ collects the primal-dual variables, and $\beta>0$
is the ADMM quadratic-penalty / step-size parameter.

Sequence acceleration for matrix iterations is a well-studied topic.
While the optimal acceleration scheme is rarely available as an analytical
expression, it is attained numerically by the Generalized Minimum
RESidual (GMRES) algorithm~\cite{saad1986GMRES}; we review this
point in detail in Section~\ref{sec:ADMM-for-Quadratic}. Using GMRES
to accelerate a matrix iteration, the resulting iterates are guaranteed
to converge as fast or faster (viewed under a particular metric) than
any sequence acceleration scheme, based on over-relaxation, momentum,
or otherwise. Hence, at a minimum, we may use GMRES to numerically
bound the amount of acceleration available using ADMM or its variants.
If convergence is sufficiently rapid, then the accelerated method
(described as Algorithm~\ref{alg:ADMM-GMRES} in Section~\ref{sec:ADMM-for-Quadratic},
which we name ADMM-GMRES) may be also be used as a solution algorithm
for (\ref{eq:ecqp}). 

\subsection{Convergence in $O(\kappa^{1/4})$ iterations}

\begin{figure}
\hfill{}\includegraphics[width=0.5\columnwidth]{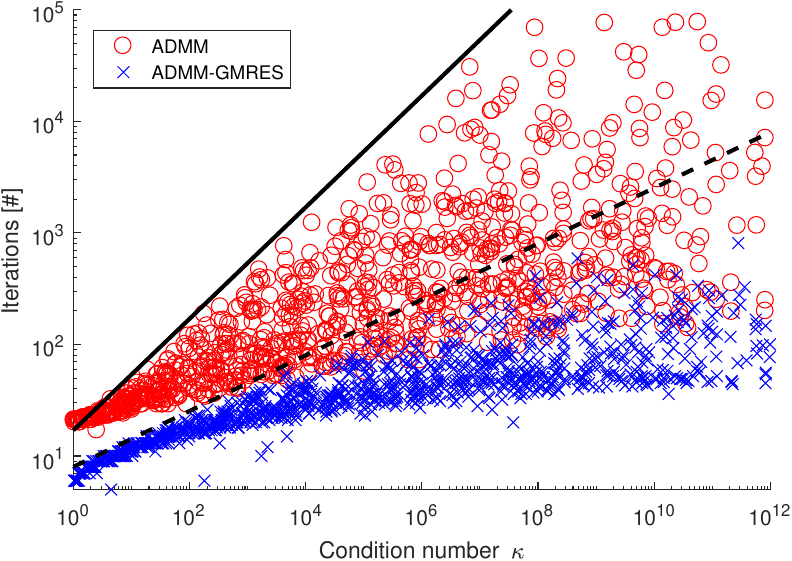}\hfill{}

\caption{\label{fig:first_comparison}Given the same 1000 randomly-generated
problems and using the same parameter choice, ADMM (circles) converges
in $O(\sqrt{\kappa})$ iterations while GMRES-accelerated ADMM (crosses)
converges in $O(\kappa^{1/4})$ iterations, where $\kappa$ is the
condition number.}
\end{figure}
Under the strong convexity assumptions described in~\cite{deng2012global},
ADMM converges at a linear rate, with a convergence rate dependent
on the parameter choice $\beta$ and the problem condition number
$\kappa$. A number of previous authors have derived the optimal fixed
parameter choice $\beta^{\star}$ that allows ADMM to converge to
an $\epsilon$-accurate solution in $O(\sqrt{\kappa}\log\epsilon^{-1})$
iterations~\cite{ghadimi2015optimal,giselsson2014diagonal,nishihara2015general}. 

Numerically accelerating this sequence using GMRES, we were surprised
to find that an $\epsilon$-accurate solution is consistently computed
in $O(\kappa^{1/4}\log\epsilon^{-1})$ iterations, for a square-root
factor reduction in the number of iterations. Fig.~\ref{fig:first_comparison}
makes this comparison for 1000 instances of (\ref{eq:ecqp}) randomly
generated using Algorithm~\ref{algr:log-normal} in Section~\ref{subsec:explain_empirical}
below. The same acceleration was also reliably observed for a wide
collection of interior-point Newton subproblems in Section~\ref{sec:Newton_dir}. 

The order-of-magnitude acceleration is surprising because the $O(\sqrt{\kappa}\log\epsilon^{-1})$
iteration bound is sharp, when considered over all $\kappa$-conditioned
(\ref{eq:ecqp}). Indeed, we give an explicit example of a ``hard''
instance of (\ref{eq:ecqp}) in Section~\ref{sec:Worst-Case-Behavior},
and prove that ADMM-GMRES cannot converge for this example at a rate
faster than $(\sqrt{\kappa}-1)/(\sqrt{\kappa}+1)$ per iteration.
This lower complexity bound is reminiscent of the famous result by
Nemirovski and Yudin~\cite{nemirovskii1983problem}, which states
first-order methods cannot minimize every smooth, strongly convex
function with condition number $\kappa$ at a rate faster than $(\sqrt{\kappa}-1)/(\sqrt{\kappa}+1)$
per iteration. Indeed, the momentum-accelerated (i.e. ``fast'')
variants of the gradient method~\cite{nesterov2004introductory}
and the projected / proximal descent method~\cite{nesterov2004introductory,nesterov2005smooth,beck2009fast,becker2011nesta}
are said to be ``optimal'' precisely because they are able to converge
at this rate.

Nevertheless, these numerical results immediately confirm the possibility
for an order-of-magnitude acceleration for ADMM under practical settings,
using over-relaxation, momentum, or some other scheme. Indeed, one
approach is to explicitly extract a $k$-parameter over-relaxation
scheme from $k$ iterations of GMRES; see~\cite{nachtigal1992hybrid}
and the references therein. Where convergence is sufficiently rapid,
we may consider using GMRES-accelerated ADMM directly as a solution
algorithm for (\ref{eq:ecqp}). In this case, the $O(k^{2})$ time
and $O(k)$ memory requirements of GMRES should be balanced against
the cost of $k$ iterations of ADMM; see our discussion in Section~\ref{subsec:restarts}.

\subsection{Main results}

Our primary goal in this paper is to understand why GMRES so frequently
achieves an order-of-magnitude acceleration for ADMM. Convergence
analysis for Krylov subspace methods is typically formulated as a
classic problem in approximation theory\textemdash how well can one
approximate the eigenvalues of the matrix $I-G(\beta)$ as the roots
of an order-$k$ polynomial? In Section~\ref{sec:k4conv}, we show
that the $\kappa^{1/4}$ factor arises because the real eigenvalues
of the matrix $I-G(\beta^{\star})$ lie along (a rescaled version
of) the interval $[1,\kappa^{1/2}]$, and the Chebyshev polynomial
approximates this entire interval with convergence factor $(\kappa^{1/4}-1)/(\kappa^{1/4}+1)$.
This is precisely the same mechanism that grants conjugate gradients
a square-root factor acceleration over basic gradient descent; see~\cite[Ch.3]{greenbaum1997iterative}
and~\cite[Ch.6.11]{saad2003iterative}. 

Within the context of ADMM-GMRES, however, the square-root factor
acceleration hinges on two additional assumptions. First, the nonsymmetric
iteration matrix $G(\beta)$ should be \emph{close to normal}, in
order for its behavior to be accurately captured by its eigenvalues.
Furthermore, $G(\beta)$ also contains \emph{complex outliers}: eigenvalues
with nonzero imaginary parts, which prevent the Chebyshev polynomial
from being directly applicable. In Section~\ref{subsec:cond}, we
prove that the $\kappa^{1/4}$ factor will persist if the complex
outliers are \emph{better conditioned} than the real eigenvalues.
Our extensive numerical trials found both assumptions to be generic
properties of (\ref{eq:ecqp}), holding for almost all problem instances.
As a consequence, convergence in $O(\kappa^{1/4})$ is observed to
be generic property for the ADMM-GMRES solution of (\ref{eq:ecqp}).
Although it is difficult to rigorously justify the assumptions, we
make a number of heuristic arguments in support of them in Section~\ref{subsec:explain_empirical}
and Section~\ref{sec:normality}.

Much of our analysis is based on the observation that ADMM applied
to (\ref{eq:ecqp}) reduces to a preconditioner for an augmented Lagrangian
version of the KKT equations for (\ref{eq:ecqp}). Within this context,
ADMM-GMRES is only one of numerous preconditioned Krylov subspace
methods available; see~\cite{benzi2005numerical} for a comprehensive
survey. Indeed, it is closely related to block-triangular~\cite{bramble1988preconditioning,zulehner2002analysis,simoncini2004block}
and augmented Lagrangian / Uzawa preconditioners~\cite{fortin1983augmented,golub2003solving}.
In Section~\ref{sec:num_res}, we compare ADMM-GMRES against some
classic preconditioners for saddle-point problems, including the block-diagonal
preconditioner, variants of the constraint preconditioner, and the
Hermitian / Skew-Hermitian splitting preconditioner. Restricting each
preconditioner to the same operations used in ADMM, we find that each
preconditioner regularly attains its worst-case iteration bound of
$O(\sqrt{\kappa})$. By comparison, the ADMM preconditioner converges
in $O(\kappa^{1/4})$ iterations for every problem considered. The
order-of-magnitude reduction in the number of iterations was able
to offset both the higher relative cost of the preconditioner, as
well as the need to deploy an expensive Krylov method like GMRES.

It remains an open question whether the order-of-magnitude acceleration
would persist for non-quadratic objectives. If the update equations
are nonlinear, then ADMM is no longer a matrix iteration, and GMRES
acceleration is no longer optimal. Nevertheless, the update equations
may be locally well-approximated by their linearization, and a nonlinear
version of GMRES like a Newton-Krylov method~\cite{brown1994convergence}
or Anderson acceleration~\cite{walker2011anderson} may prove to
be useful.

\subsection{\label{subsec:applic_SDP}Applications to large-scale semidefinite
programming}

The original motivation of this paper is to solve semidefinite programs
(SDPs)
\begin{equation}
\underset{Y\succeq0}{\text{minimize }}\tr QY\text{ subject to }\tr B_{i}Y=p_{i}\text{ for all }i\in\{1,\ldots,m\},\tag{SDP}\label{eq:SDP}
\end{equation}
in which the problem data $Q,B_{1},\ldots,B_{m}$ are assumed to be
\emph{large-and-sparse}. Here, each matrix is $\theta\times\theta$
real symmetric, $Y\succeq0$ indicates that $Y$ is symmetric positive
semidefinite. Semidefinite programs are usually solved using an interior-point
method; the associated computation cost is dominated by the Newton
subproblem 
\begin{alignat}{2}
 & \text{minimize } & \frac{1}{2}\|W^{1/2}(X-\hat{X})W^{1/2}\|_{F}^{2}+p^{T}z\tag{NEWT}\label{eq:newt}\\
 & \text{subject to } & X+\sum_{i=1}^{m}z_{i}B_{i}=Q,\nonumber 
\end{alignat}
which must be solved at each interior-point iteration, using iteration-
and algorithm-specific $\theta\times\theta$ matrices $W$ and $\hat{X}$.
In practice, convergence to machine precision almost always occurs
within 30-50 iterations, so it is helpful to view the ``practical
complexity'' of (\ref{eq:SDP}) as a modest constant times the cost
of solving (\ref{eq:newt}).

General-purpose interior-point methods solve (\ref{eq:newt}) directly,
by performing Gaussian elimination on its linear KKT equations. An
important feature of interior-point methods for SDPs is that these
KKT equations are typically dense, despite any sparsity in the data
matrices $Q,B_{1},\ldots,B_{m}$. As a consequence, the direct approach
solves (\ref{eq:newt}) in approximately the same amount of time and
memory for highly sparse instances as it does for fully dense ones.
In the case of larger problems with $\theta$ and $m$ on the order
of thousands, even fitting the KKT equations into memory becomes very
difficult. 

Alternatively, we can solve (\ref{eq:newt}) using an iterative algorithm,
like conjugate gradients (CG)~\cite{karmarkar1984new,vandenberghe1995primal,vandenberghe1996semidefinite,toh2002solving}
or ADMM~\cite{pakazad2014distributed}, as a set of inner iterations
within an outer interior-point method. These iterative algorithms
have low per-iteration and memory costs that can be further reduced
by exploiting sparsity. On the other hand, the resulting interior-point
method suffer from ``diminishing returns'' typical of first-order
methods: more and more inner iterations are required for each additional
outer iteration. Formally, (\ref{eq:newt}) has condition number $\kappa=\Theta(1/\epsilon^{2})$
at an $\epsilon$-accurate interior-point iterate, so taking the standard
$O(\sqrt{\kappa}\log\epsilon^{-1})$ iteration bound for CG and ADMM
and the $O(\log\epsilon^{-1})$ iteration bound for the interior-point
method, we find that the combined algorithm converges to an $\epsilon$-accurate
solution of (\ref{eq:SDP}) in $O(\epsilon^{-1}(\log\epsilon^{-1})^{2})$
inner iterations. Up to a logarithmic factor, this is the same $O(1/\epsilon)$
iteration bound as obtained by applying a standard first-order method
directly to (\ref{eq:SDP}); see~\cite{nesterov2007smoothing,wen2010alternating}.
There seems to be little justification for the added complexities
of the interior-point method.

Now, suppose that GMRES-accelerated ADMM is able to solve (\ref{eq:newt})
to $\epsilon$-accuracy in $O(\kappa^{1/4}\log\epsilon^{-1})$ iterations.
Then, embedding this accelerated ADMM within an outer interior-point
method yields a combined algorithm that converges to an $\epsilon$-accurate
solution of (\ref{eq:SDP}) in $O(\epsilon^{-1/2}(\log\epsilon^{-1})^{2})$
inner iterations. Up to a logarithmic factor, this is the same $O(1/\sqrt{\epsilon})$
iteration bound as ``fast'' first-order algorithms like Nesterov's
accelerated gradient method~\cite{nesterov1983method}, FISTA~\cite{beck2009fast},
and NESTA~\cite{becker2011nesta}. The order-of-magnitude acceleration
previously described for the smooth, strongly convex quadratic problem
(\ref{eq:ecqp}) has been extended to the nonsmooth, weakly-convex,
non-quadratic problem (\ref{eq:SDP}) through the use of an outer
interior-point loop.

Section~\ref{sec:Newton_dir} tests this idea by embedding ADMM-GMRES
within SeDuMi~\cite{sturm1999using}, a popular open-source interior-point
method written in MATLAB, and using it to solve problems from SDPLIB~\cite{borchers1999SDPLIB}
and the Seventh DIMACS Implementation Challenge~\cite{pataki2002dimacs}.
In our results, ADMM-GMRES does indeed solve every instance of (\ref{eq:newt})
in $O(\kappa^{1/4}\log\epsilon^{-1})$ iterations, though the costs
associated with GMRES become too high past $\approx30$ iterations.
Restarting ADMM-GMRES every 25 iterations yielded a ``fast'' first-order
method that converged in $O(1/\sqrt{\epsilon})$ iterations for 8
out of the 10 DIMACS problems considered.

\subsection{Related work}

\subsubsection*{Accelerating ADMM}

When applied to a weakly convex, possibly nonsmooth problem, ADMM
has a sublinear error rate, converging to an $\epsilon$-accurate
solution within $O(1/\epsilon)$ iterations~\cite{he20121}. Most
existing work on accelerating ADMM~\cite{goldstein2014fast,ouyang2015accelerated,kadkhodaie2015accelerated}
aim to improve the iteration bound to $O(1/\sqrt{\epsilon})$ by assuming
strong convexity and either applying the momentum ideas of Nesterov~\cite{nesterov1983method,nesterov2005smooth}
and Beck and Teboulle~\cite{beck2009fast}, or by adopting a ``fast''
version of the Douglas\textendash Rachford algorithm~\cite{esser2010general,chambolle2011first,patrinos2014douglas},
noting that ADMM is just the Douglas\textendash Rachford algorithm
applied to the dual problem~\cite{gabay1983applications}. 

For smooth and strongly convex problems, ADMM converges at a linear
rate, producing an $\epsilon$-accurate solution in $O(\sqrt{\kappa}\log\epsilon^{-1})$
iterations~\cite{deng2012global}. Within this regime, the most widely
used acceleration technique is simple over-relaxation, which reliably
reduces total iteration count by a small constant~\cite{wen2010alternating,giselsson2014diagonal,ghadimi2015optimal,nishihara2015general}. 

\subsubsection*{ADMM for SDPs}

Semidefinite programs are nonsmooth and weakly convex by construction.
While ADMM (or an accelerated variant) can be directly applied to
solve (\ref{eq:SDP})~\cite{wen2010alternating,odonoghue2016conic},
the resulting iterations converge at a sublinear rate in the worst-case,
requiring up to $O(1/\epsilon)$ iterations to produce an $\epsilon$-accurate
iterate. In practice, ADMM often performs much better than its worst-case,
producing 6-7 accurate digits in just 200-500 iterations over a wide
array of test problems~\cite{wen2010alternating}. Nevertheless,
the algorithm does regularly attain its worst-case bound; several
thousand iterations may be required to produce just 2-3 accurate digits~\cite{madani2015admm,kalbat2015fast}. 

In this paper, we apply ADMM to the inner Newton subproblem (\ref{eq:newt})
within an outer interior-point solution of (\ref{eq:SDP}). Using
GMRES to accelerate ADMM, the number of iterations to $\epsilon$-accuracy
is consistently reduced to $O(1/\sqrt{\epsilon})$ iterations, for
a square-root factor improvement over directly applying ADMM to (\ref{eq:SDP}).
The $O(1/\epsilon)$ worst-case iteration bound remains unchanged,
though our experimental results suggest that the worst-case is difficult
to attain.

The per-iteration complexity of ADMM applied to either (\ref{eq:SDP})
or (\ref{eq:newt}) is at least cubic $\Omega(\theta^{3})$ time and
quadratic $\Omega(\theta^{2})$ space, due to the explicit storage
of, and algebraic manipulations with, a fully-dense $\theta\times\theta$
matrix variable $Y$. These complexity figures limit the algorithm
to medium-sized SDPs, with $\theta$ on the order of a few thousand.
However, if the data matrices $Q,B_{1},\ldots,B_{m}$ are sparse with
a \emph{chordal} aggregate sparsity pattern\footnote{An equivalent condition is to say that every $\theta\times\theta$
dual variable $S=Q-\sum_{i}z_{i}B_{i}\succ0$ can symmetrically permuted
and factored using Cholesky factorization in linear $O(\theta)$ time
and space.}, then it is possible to represent the fully-dense $\theta\times\theta$
matrix variable $Y$ using $\theta$ fully-dense $O(1)\times O(1)$
matrix variables $Y_{1},\ldots,Y_{\theta}$ via a \emph{positive semidefinite
matrix completion} argument; see~\cite{fukuda2001exploiting,nakata2003exploiting}
and also~\cite{vandenberghe2015chordal} for an extensive survey.
This technique, named \emph{clique tree conversion} or \emph{chordal
conversion}, reduces the per-iteration complexity of ADMM to linear
$O(\theta)$ time and space, thereby making it suitable for large-scale
SDPs with $\theta$ on the order of tens of thousands~\cite{madani2015admm,zheng2016fast}.
Unfortunately, many large-scale SDPs do not satisfy the chordal sparsity
property; for example, it is not satisfied by the large-scale test
problems in Section~\ref{sec:Newton_dir}.

\subsubsection*{Iterative solution of the Newton subproblem}

The idea of applying a preconditioned iterative solver to the interior-point
Newton subproblem dates back to the original Karmarkar interior-point
method~\cite{karmarkar1984new}, and remains the standard approach
large-scale $\ell_{1}$-regularized regression~\cite{candes2005l1magic}
and network flow problems~\cite[Chapter 4]{mitchell1998interior}.
Preconditioners based on sparse matrix ideas, such as incomplete Cholesky
factorizations~\cite{lin1999incomplete} and constraint preconditioning~\cite{bergamaschi2004preconditioning},
work very well for linear and quadratic programs, but have been less
successful for SDP, primarily due to the density of their associated
Newton subproblem; see the discussions in~\cite{toh2002solving,toh2004solving}
for further details. 

Preconditioners based on the spectral properties of the Newton subproblem
(\ref{eq:newt}), such as the projection preconditioner~\cite{toh2002solving,toh2004solving},
and the partial Cholesky preconditioner~\cite{gondzio2012matrix,bellavia2013matrix},
work very well for SDPs. Their key insight is to note that the number
of \emph{active inequality constraints} in (\ref{eq:SDP}) determines
the number of ill-conditioned dimensions in its Newton subproblem
(\ref{eq:newt}). If only a few constraints are active, or equivalently,
if the solution $Y^{\star}$ to (\ref{eq:SDP}) is \emph{low-rank},
then the number of ``bad'' dimensions in (\ref{eq:newt}) is small,
and can be corrected using a low-rank perturbation. The resulting
preconditioned problem has a bounded condition number $\kappa=O(1)$
at every $\epsilon$-accurate interior-point iterate. 

This paper suggests ADMM as a preconditioner for the Newton subproblem
(\ref{eq:newt}). The resulting preconditioned problem has condition
number $\kappa=O(1/\sqrt{\epsilon})$ at an $\epsilon$-accurate interior-point
iterate under mild technical assumptions. This is a weaker guarantee
than the $O(1)$ figure for the spectral preconditioners described
above. However, it remains applicable even when the solution $Y^{\star}$
to (\ref{eq:SDP}) is not low-rank, and the number of ``bad'' dimensions
in (\ref{eq:newt}) is not small. In comparison, the spectral preconditioners
are no longer efficient under this regime, and may have costs comparable
to that of directly solving the full problem. 

Given an arbitrary fully-dense $\theta\times\theta$ algorithm matrix
$W$, the per-iteration complexity of an iterative solver applied
to (\ref{eq:newt}) is at least cubic $\Omega(\theta^{3})$ time and
quadratic $\Omega(\theta^{2})$ space in general. However, if a \emph{dual-scaling}
interior-point method is used (over a primal-dual interior-point method),
then $W$ inherits an additional structure: its inverse matrix $W^{-1}$
is \emph{sparse,} with the same aggregate sparsity pattern as the
problem data $Q,B_{1},\ldots,B_{m}$~\cite{benson2000solving,vandenberghe2015chordal,Bellavia2018}.
Moreover, if the data matrices $Q,B_{1},\ldots,B_{m}$ are sparse
with a \emph{chordal} aggregate sparsity pattern, then $W^{-1}$ can
be efficiently factored in linear $O(\theta)$ time and space~\cite{dahl2008covariance,andersen2013logarithmic,zhang2018large}.
Implicitly representing $W$ via its sparse inverse $W^{-1}$ (possibly
in factored form) reduces the per-iteration cost of iterative solvers
to linear $O(\theta)$ time and space, thereby making them available
to large-scale problems with $\theta$ on the order of tens to hundreds
of thousands~\cite{Bellavia2018,zhang2018large}. 

\subsection{Notation}

Throughout this paper, we will frequently refer to the condition number
$\kappa$, the gradient Lipschitz constant $L$, the strong convexity
parameter $\mu$, which are defined in terms of a rescaled objective
matrix $\tilde{D}$ in (\ref{eq:mu_L_kap}). The positive integer
$k$ is reserved for the iteration index. Accuracy is always measured
with respect to the metric $\|\cdot\|_{M}$ defined in (\ref{eq:gmres_cost}).

The positive integers $n,$ $m,$ $\ell$ refer to the dimensions
of the primal-dual variables $x\in\R^{n},$ $z\in\R^{m},$ and $y\in\R^{\ell}$
respectively. Their sum $N=n+m+\ell$ denotes the dimension of the
concatenated variable $u=[x;z;y]\in\R^{N}$. In discussing interior-point
methods for semidefinite programs in Section~\ref{sec:Newton_dir},
the positive integer $\theta$ denotes the order of the semidefinite
cone.

Our notation is otherwise standard, with the following exceptions.
We use $\Lambda\{A\}$ to refer to the eigenvalues of $A$. The set
$\P_{k}$ denotes the space of order-$k$ polynomials. Given a polynomial
$p(z)$, we denote its maximum modulus over a compact subset of the
complex plane $\mathcal{S}\subset\C$ as $\|p(z)\|_{\mathcal{S}}=\max_{z\in\mathcal{S}}|p(z)|$. 

\section{ADMM for quadratic problems\label{sec:ADMM-for-Quadratic}}

Beginning with a choice of the quadratic-penalty / step-size parameter
$\beta>0$ and initial points $u^{0}=[x^{0};z^{0};y^{0}]$, ADMM applied
to (\ref{eq:ecqp}) generates iterates \begin{subequations}\label{eq:admm_iters}
\begin{align}
x^{k+1} & =\arg\min_{x}\frac{1}{2}x^{T}Dx+c^{T}x+\frac{\beta}{2}\|Ax+Bz^{k}-c+y^{k}\|^{2},\label{eq:x_update}\\
z^{k+1} & =\arg\min_{z}p^{T}z+\frac{\beta}{2}\|Ax^{k+1}+Bz-c+y^{k}\|^{2},\label{eq:z_update}\\
y^{k+1} & =y^{k}+(Ax^{k+1}+Bz^{k+1}-c).\label{eq:y_update}
\end{align}
\end{subequations}Some basic algebraic manipulations reveal these
to be a matrix-splitting iteration for the (rescaled) augmented Lagrangian
KKT system (see e.g.~\cite{fortin1983augmented,golub2003solving})
\begin{equation}
\underbrace{\begin{bmatrix}\beta^{-1}D+A^{T}A & A^{T}B & A^{T}\\
B^{T}A & B^{T}B & B^{T}\\
A & B & 0
\end{bmatrix}}_{H(\beta)}\underbrace{\begin{bmatrix}x\\
z\\
y
\end{bmatrix}}_{u}=\underbrace{\begin{bmatrix}A^{T}d-\beta^{-1}c\\
B^{T}d-\beta^{-1}p\\
d
\end{bmatrix}}_{v(\beta)},\label{eq:augKKT}
\end{equation}
using a Gauss-Seidel\textendash like splitting
\begin{equation}
H(\beta)=\underbrace{\begin{bmatrix}\beta^{-1}D+A^{T}A & 0 & 0\\
B^{T}A & B^{T}B & 0\\
A & B & -I
\end{bmatrix}}_{M(\beta)}-\underbrace{\begin{bmatrix}0 & -A^{T}B & -A^{T}\\
0 & 0 & -B^{T}\\
0 & 0 & -I
\end{bmatrix}}_{N(\beta)}.\label{eq:augKKTsplit}
\end{equation}
Indeed, fixing the parameter $\beta$, (\ref{eq:admm_iters}) is precisely
the linear fixed-point iterations
\begin{equation}
u^{k+1}=\underbrace{M(\beta)^{-1}N(\beta)}_{G(\beta)}u^{k}+\underbrace{M(\beta)^{-1}v(\beta)}_{b(\beta)}.\label{eq:fixed_point_iter}
\end{equation}

ADMM assumes that \emph{black-box oracles} are available for evaluating
matrix-vector products with $(\beta^{-1}D+A^{T}A)^{-1}$, $(B^{T}B)^{-1}$,
$A$, $B$, $A^{T}$, and $B^{T}$, without necessarily requiring
explicit access to these matrices. The method can be effective only
if the oracles are not too expensive to set-up and call. In practice,
efficient \emph{matrix-implicit} implementations frequently arise
through problem-specific structure. For example, when the matrices
$A$, $B$, and $D$ are large-and-sparse, the matrices $\beta^{-1}D+A^{T}A$
and $B^{T}B$ often admit sparse Cholesky factorizations. After precomputing
the factorizations in linear-time, each matrix-vector product with
$(\beta^{-1}D+A^{T}A)^{-1}$ or $(B^{T}B)^{-1}$ may be evaluated
in linear-time, as the solution of two sparse, triangular linear systems.
Alternative structures can also be exploited, including the Kronecker
factorization $D=W\otimes W$ for the SDP Newton subproblem in Section~\ref{subsec:Implementation},
as well as Toeplitz / Hankel / Circulant matrix structures. For further
discussion on these implementation issues, we direct the interested
reader to~\cite[Sec.4.2]{boyd2011distributed}. 

\subsection{Sequence acceleration}

Let us view ADMM as a black box $T_{\beta}(\cdot)$ that maps a given
test point $u$ to its image $T_{\beta}(u)$, in order to rewrite
the basic ADMM method as the iterated map 
\begin{equation}
u^{k+1}=T_{\beta}(u^{k}).\label{eq:iter_map}
\end{equation}
Under strong convexity assumptions, (\ref{eq:iter_map}) converges
linearly to a unique fixed-point~\cite{deng2012global}.
\begin{assume}
\label{ass:regularity}The matrix $D$ is symmetric positive definite,
the matrix $A$ has full row-rank (i.e. $AA^{T}$ is invertible),
and the matrix $B$ has full column-rank (i.e. $B^{T}B$ is invertible). 
\end{assume}
\begin{rem}
Given that $A\in\R^{\ell\times n}$ and $B\in\R^{\ell\times m}$,
Assumption~\ref{ass:regularity} can only be satisfied if $n\ge\ell\ge m$.
(Recall that $x\in\R^{n}$, $y\in\R^{\ell}$, $z\in\R^{m}$.)
\end{rem}
The convergence rate of (\ref{eq:iter_map}) depends on parameter
choice $\beta$ and the problem condition number $\kappa$, defined
as the ratio of a gradient Lipschitz parameter $L$ and a strong convexity
parameter $\mu$,
\begin{equation}
\tilde{D}\triangleq(AD^{-1}A^{T})^{-1},\quad\mu\triangleq\lambda_{\min}(\tilde{D}),\quad L\triangleq\lambda_{\max}(\tilde{D}),\quad\kappa\triangleq L/\mu.\label{eq:mu_L_kap}
\end{equation}
In particular, the parameter choice $\beta=\sqrt{\mu L}$ allows an
$\epsilon$-accurate solution to be computed in no more than $O(\sqrt{\kappa}\log\epsilon^{-1})$
iterations~\cite{ghadimi2015optimal,giselsson2014diagonal,nishihara2015general}. 

Sequence acceleration seeks to find an $\epsilon$-accurate approximation
of the fixed point $u^{\star}=T_{\beta}(u^{\star})$ while making
as few calls to the black-box oracle as possible. Two popular approaches
are over-relaxation, which linearly extrapolates the current step,
\begin{equation}
u^{k+1}=u^{k}+\omega_{k}[T_{\beta}(u^{k})-u^{k}],\label{eq:over-relax}
\end{equation}
and momentum, which linearly extrapolates the previous step, 
\begin{equation}
\tilde{u}=u^{k}+\theta_{k}(u^{k}-u^{k-1}),\qquad u^{k+1}=T_{\beta}(\tilde{u}).\label{eq:momentum}
\end{equation}
In each case, the $(k+1)$-th iterate is selected from the plane that
crosses the initial point and the $k$ images
\begin{equation}
u^{k+1}\in\mathrm{Aff}\{u^{0},T_{\beta}(u^{0}),T_{\beta}(u^{1}),\ldots,T_{\beta}(u^{k})\}.\label{eq:seach_space}
\end{equation}
The affine hull linearly extrapolates the information collected from
$k$ evaluations of the black-box oracle; the parameters $\theta_{1},\ldots,\theta_{k}$
and $\omega_{1},\ldots,\omega_{k}$ may be viewed as its coordinates.
By carefully tuning these parameters, it is possible to select \emph{better}
candidates than the default choice produced by the iterated map (\ref{eq:iter_map}),
thereby yielding a sequence $u^{0},u^{1},\ldots,u^{k}$ with an accelerated
convergence rate.

\subsection{The optimality of GMRES}

Let the black-box oracle $T_{\beta}(\cdot)$ be affine, meaning that
there exists some matrix $G(\beta)$ and vector $b(\beta)$ such that
$T_{\beta}(u)=G(\beta)u+b(\beta)$. Furthermore, let us measure the
accuracy of a test point $u$ using an implicit but easily computable
Euclidean metric\footnote{Indeed, $\kappa_{M}^{-1}\le\|u-u^{\star}\|_{M}/\|u-u^{\star}\|\le\kappa_{M}$,
where $\kappa_{M}=\mathrm{cond}(I-G(\beta))$ is finite because the
ADMM iterations converge to a unique fixed-point.}
\begin{equation}
\|u-u^{\star}\|_{M}\triangleq\|u-[G(\beta)u+b(\beta)]\|=\|[I-G(\beta)](u-u^{\star})\|.\label{eq:gmres_cost}
\end{equation}
Then, the affine search space (\ref{eq:seach_space}) reduces to a
Krylov subspace

\begin{equation}
u^{k+1}\in u^{0}+\mathrm{span}\{r,Gr,\ldots,G^{k}r\}\label{eq:krylov_def}
\end{equation}
where $G\equiv G(\beta)$ and $r=u^{0}-[G(\beta)u^{0}+b(\beta)]$,
and the problem of selecting the \emph{best} candidate from the Krylov
subspace (\ref{eq:krylov_def}) is numerically solved by GMRES. We
defer to standard texts for its implementation details, e.g.~\cite[Alg. 6.9]{saad2003iterative}
or \cite[Alg. 6.10]{saad2003iterative}, and only note that the algorithm
can be viewed as a ``black-box'' that solves the following projected
least-squares problem at the $k$-th iteration:
\begin{align}
\GMRES_{k}(A,b) & =\arg\min\left\{ \|b-Ax\|_{2}:x\in\mathrm{span}\{b,Ab,\ldots,A^{k-1}b\}\right\} ,\label{eq:gmres_min}
\end{align}
in $\Theta(k^{2}N)$ flops, $\Theta(kN)$ memory, and $k$ matrix-vector
products with $A$. Consider the following algorithm.
\begin{algor}[ADMM-GMRES]
\label{alg:ADMM-GMRES}Input: The update operator $T_{\beta}$ that
implements (\ref{eq:admm_iters}) as $u^{k+1}=T_{\beta}(u^{k})$;
Initial point $u^{0}$; Number of iterations $k$.
\begin{enumerate}
\item Precompute and store $T_{\beta}(u^{0})=G(\beta)u^{0}+b(\beta)$ and
$r=u^{0}-T_{\beta}(u^{0})$;
\item Call $\Delta u=\GMRES_{k}(I-G(\beta),\,r)$, while evaluating each
matrix-vector product as $[I-G(\beta)]h=h-[T_{\beta}(u_{0}+h)-T_{\beta}(u^{0})]$;
\item Output $u^{k}=u^{0}-\Delta u$.
\end{enumerate}
\end{algor}
The optimality of GMRES in (\ref{eq:gmres_min}) guarantees $\|u^{k}-u^{\star}\|_{M}$
to be smaller than that of regular ADMM, as well as any accelerated
variant that selects its $k$-th iterate $u^{k}$ from (\ref{eq:seach_space}).
In other words, no linearly\emph{ }extrapolating sequence acceleration
scheme, based on momentum, over-relaxation, or otherwise, can converge
faster than GMRES when viewed under this metric. 

\subsection{ADMM as a preconditioner}

The fixed-point equation associated with the ADMM iterations (\ref{eq:fixed_point_iter})
\begin{equation}
u^{\star}-G(\beta)u^{\star}=b(\beta),\label{eq:admm_fixed}
\end{equation}
is a linear system of equations when $\beta$ is held fixed, which
can be solved using GMRES. In the previous subsection, we named the
resulting method ADMM-GMRES, and viewed it as an optimally accelerated
version of ADMM. Equivalently, (\ref{eq:admm_fixed}) is also the
\emph{left-preconditioned} system of equations
\begin{equation}
M^{-1}(\beta)[H(\beta)u^{\star}-v(\beta)]=0\qquad\Leftrightarrow\qquad\text{(\ref{eq:admm_fixed})},\label{eq:left-prec}
\end{equation}
where $H$ and $v$ comprise the augmented Lagrangian KKT system in
(\ref{eq:augKKT}), and $M$ is the preconditioner matrix defined
in (\ref{eq:augKKTsplit}). Note that the ADMM iteration matrix satisfies
$G(\beta)=I-M^{-1}(\beta)H(\beta)$ by definition. In turn, ADMM-GMRES
is equivalent to a preconditioned GMRES solution of the augmented
KKT system $H(\beta)u=v(\beta)$ using $M(\beta)$ as the preconditioner. 

\subsection{\label{subsec:restarts}Reducing the cost of GMRES}

A significant shortcoming of GMRES is its need to store and manipulate
an $N\times k$ dense matrix at the $k$-th iteration. Its $\Theta(k^{2}N)$
time and $\Theta(kN)$ memory requirements become unsustainable once
$k$ grows large. It was proved by Faber and Manteuffel that these
complexity figures cannot be substantially reduced without destroying
the optimal property of GMRES~\cite{faber1984necessary}. Hence,
if many GMRES iterations are desired, then we must give up on its
optimality and adopt a limited-memory heuristic.

One limited-memory approach is to periodically restart GMRES: after
$p$ iterations, the final iterate is extracted, and used as the initial
point for a new set of $p$ iterations. An issue with the resulting
algorithm, known as GMRES($p$), is that it can stall, meaning that
it may fail to make further progress after a certain number of iterations.
Alternatively, a Krylov method based on Lanczos biorthogonalization
may be used, including BiCG, QMR and their transpose-free variants;
see~\cite{saad2003iterative}. These are entirely heuristic, but
tend to work well when GMRES converges quickly, and do not stall as
easily as restarted GMRES. QMR is often preferred over BiCG for being
more stable and for having a convergence analysis somewhat related
to GMRES. 

Limited-memory variants of ADMM-GMRES can be developed by viewing
it as a preconditioned GMRES solution of the augmented KKT system
$H(\beta)u=v(\beta)$ in (\ref{eq:augKKT}) using the matrix $M(\beta)$
in (\ref{eq:augKKTsplit}) as the preconditioner. For example, we
obtain ADMM-GMRES($p$) and ADMM-QMR by using GMRES($p$) and QMR
in place of regular GMRES, respectively. Fig.~\ref{fig:restarts}
compares ADMM with ADMM-GMRES and these two limited-memory variants,
as applied to two problems of comparable conditioning. In the first,
``easy'' example, all limited-memory variants of ADMM-GMRES outperform
basic ADMM, with ADMM-QMR converging almost as rapidly as ADMM-GMRES,
despite requring far less time and memory. But in the second, ``difficult''
example, all of these limited-memory variants stall, or get close
to stalling. Only ADMM-GMRES is able to converge at the desired $(\kappa^{1/4}-1)/(\kappa^{1/4}+1)$
rate.
\begin{figure}
\hfill{}\subfloat[]{\includegraphics[width=0.45\columnwidth]{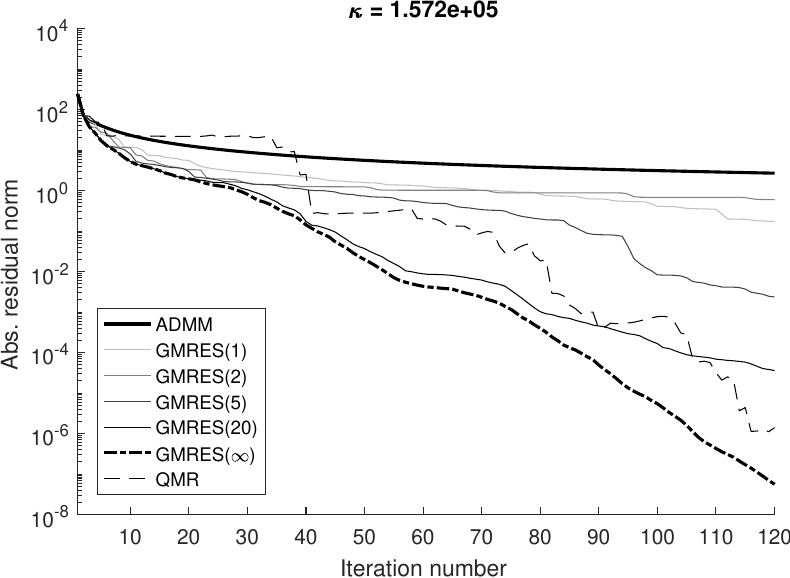}

}\hfill{}\subfloat[]{\includegraphics[width=0.45\columnwidth]{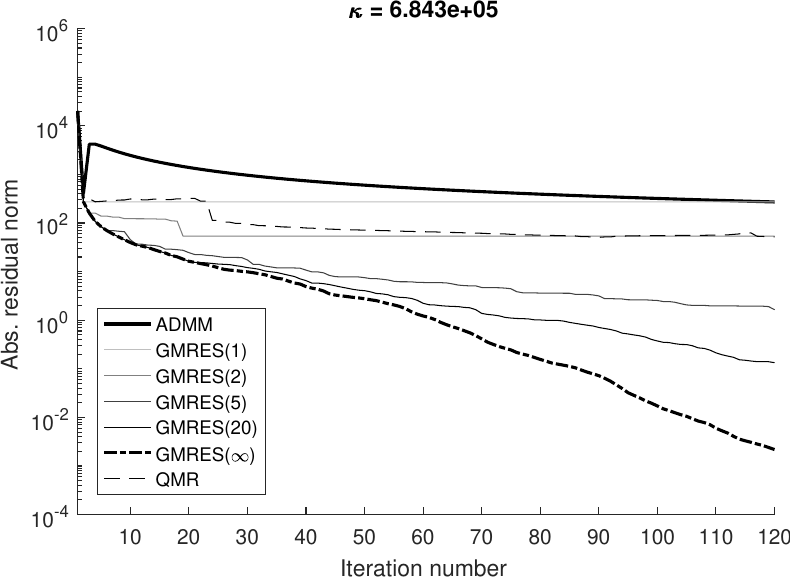}

}\hfill{}

\caption{\label{fig:restarts}Comparison of ADMM with ADMM-GMRES and its limited-memory
variants: (a) an ``easy'' example; (b) a ``difficult'' example. }
\end{figure}

\section{GMRES convergence analysis}

Krylov subspace methods like GMRES are closely associated with the
idea of approximating a matrix inverse using a low-order matrix polynomial.
To explain, consider substituting the constraints of the GMRES least-squares
problem (\ref{eq:gmres_min}) into its objective, to yield
\begin{equation}
\|b-Ax^{k}\|=\min_{\alpha_{1},\ldots,\alpha_{k}}\left\Vert b-A\left(\sum_{j=1}^{k}\alpha_{j}A^{j-1}b\right)\right\Vert =\min_{p\in\P_{k-1}}\|[I-Ap(A)]b\|.\label{eq:gmres_poly}
\end{equation}
If there exists an accurate low-order matrix polynomial approximation
$p(A)$ for the matrix inverse $A^{-1}$, then $Ap(A)\approx I$ and
$\|[I-Ap(A)]b\|\approx0$. GMRES must converge rapidly, since it optimizes
over all polynomials in (\ref{eq:gmres_poly}). 

Equivalently, we may solve the\emph{ residual minimization} problem
\begin{equation}
\min_{p\in\P_{k-1}}\left\{ \|q(A)b\|:q(z)=1-zp(z)\right\} =\min_{\begin{subarray}{c}
q\in\P_{k}\\
q(0)=1
\end{subarray}}\|q(A)b\|,\label{eq:gmres_poly2}
\end{equation}
and recover $p(A)$ via $p(z)=z^{-1}(q(z)-1)$. Equation (\ref{eq:gmres_poly2})
is the standard tool for analyzing the convergence of optimal Krylov
subspace methods like GMRES; see~\cite{greenbaum1997iterative,driscoll1998potential,saad2003iterative}.
For each $k\in\{1,2,\ldots\}$, the typical proof constructs a heuristic
polynomial $h_{k}(z)\in\P_{k}$ satisfying $h_{k}(0)=1$, and demonstrates
that the induced 2-norm of the matrix polynomial $h_{k}(A)$ converges
geometrically, as in $\|h_{k}(A)\|\le\alpha\varrho^{k}$. Then, since
GMRES optimizes over all polynomials in (\ref{eq:gmres_poly2}), it
must converge at least as quickly as this particular choice of polynomial:
\begin{equation}
\min_{\begin{subarray}{c}
q\in\P_{k}\\
q(0)=1
\end{subarray}}\|q(A)b\|\le\|h_{k}(A)b\|\le\|b\|\alpha\varrho^{k}.
\end{equation}
Hence, we conclude that GMRES converges at the asymptotic rate of
$\varrho$, requiring no more than $(1-\varrho)^{-1}\log(\alpha\epsilon^{-1})$
iterations to converge to an $\epsilon$-accurate iterate.

In this section, we will simplify the residual minimization problem
(\ref{eq:gmres_poly2}) associated with the matrix $A\gets[I-G(\beta)]$
to an easier one associated with the $\ell\times\ell$ nonsymmetric
matrix 
\begin{equation}
K(\beta)\triangleq\begin{bmatrix}Q^{T}\\
-P^{T}
\end{bmatrix}\left[(\beta^{-1}\tilde{D}+I)^{-1}-(\beta\tilde{D}^{-1}+I)^{-1}\right]\begin{bmatrix}Q & P\end{bmatrix},\label{eq:Kdef}
\end{equation}
in which $\tilde{D}=(AD^{-1}A^{T})^{-1}$ and $B=\left[\begin{smallmatrix}Q & P\end{smallmatrix}\right]\left[\begin{smallmatrix}R\\
0
\end{smallmatrix}\right]$ is the QR decomposition of $B$. (Recall that $\ell$ is the number
of rows in $A$, $B$, as well as the dimension of the Lagrange multiplier
$y$.) We begin by showing that only $\ell\le\frac{1}{2}N$ eigenvalues
of the $N\times N$ iteration matrix $G(\beta)$ are nonzero, and
these have values that coincide with the eigenvalues of $G_{22}(\beta)=\frac{1}{2}I+\frac{1}{2}K(\beta)$. 
\begin{lemma}
\label{lem:blkschur_ad}Let $B=\left[\begin{smallmatrix}Q & P\end{smallmatrix}\right]\left[\begin{smallmatrix}R\\
0
\end{smallmatrix}\right]$ be the QR decomposition of $B$. Then the ADMM iteration matrix $G(\beta)$
defined in (\ref{eq:fixed_point_iter}) has a block decomposition
\begin{equation}
G(\beta)=\left[\begin{array}{c|cc|c}
I_{n} & 0 & 0 & 0\\
0 & R^{-1} & 0 & 0\\
0 & 0 & P & Q
\end{array}\right]\left[\begin{array}{c|c|c}
0_{n} & G_{12}(\beta) & G_{13}(\beta)\\
\hline 0 & G_{22}(\beta) & G_{23}(\beta)\\
\hline 0 & 0 & 0_{m}
\end{array}\right]\left[\begin{array}{ccc}
I_{n} & 0 & 0\\
\hline 0 & R & 0\\
0 & 0 & P^{T}\\
\hline 0 & 0 & Q^{T}
\end{array}\right],\label{eq:AD_schur_blks}
\end{equation}
into blocks of size $n$, $\ell$, and $m$ respectively, in which
the constituent matrices are 
\begin{gather*}
G_{12}(\beta)=-\beta D^{-1}A^{T}(\beta^{-1}\tilde{D}+I)^{-1}\begin{bmatrix}Q & P\end{bmatrix},\\
G_{13}(\beta)=-\beta D^{-1}A^{T}(\beta^{-1}\tilde{D}+I)^{-1}Q,\\
G_{22}(\beta)=\frac{1}{2}I+\frac{1}{2}K(\beta),\qquad G_{23}(\beta)=\begin{bmatrix}Q^{T}\\
-P^{T}
\end{bmatrix}(\beta\tilde{D}^{-1}+I)^{-1}Q,
\end{gather*}
$K(\beta)$ is defined in (\ref{eq:Kdef}), and $\tilde{D}=(AD^{-1}A^{T})^{-1}$.
\end{lemma}
\begin{proof}
Follows from direct computation and applications of the Sherman\textendash Morrison\textendash Woodbury
identity.
\end{proof}
\begin{rem}
\label{rem:Knrm}The special structure of $K(\beta)$ allows its singular
values to be explicitly stated. In particular, 
\begin{equation}
\|K(\beta)\|=\frac{\gamma-1}{\gamma+1}\quad\text{where }\gamma=\max\left\{ \frac{L}{\beta},\frac{\beta}{\mu}\right\} ,\label{eq:Knrm}
\end{equation}
so the eigenvalues of $K(\beta)$ are contained within the disk on
the complex plane centered at at the origin, with radius $(\gamma-1)/(\gamma+1)$.
This is an important characterization of $K(\beta)$ that we will
use extensively in Sections~\ref{sec:Worst-Case-Behavior}~\&~\ref{sec:k4conv}.
\end{rem}
Furthermore, the block pattern in the Schur decomposition (\ref{eq:AD_schur_blks})
suggests that the Jordan block associated with each zero eigenvalue
of $G(\beta)$ is at most size $2\times2$. After two iterations,
ADMM becomes entirely dependent upon the inner iteration matrix $G_{22}(\beta)$.
\begin{cor}
\label{cor:poly_2norm}Given $\beta>0$ and any polynomial $p(\cdot)$,
we have
\[
\|p(G(\beta))\,G^{2}(\beta)\|\le c_{0}\left(1+\beta\|D^{-1}A^{T}\|\right)\|p(G_{22}(\beta))\|,
\]
where $c_{0}=2\max\{\sigma_{\max}(B),1/\sigma_{\min}(B),\sigma_{\max}(B)/\sigma_{\min}(B)\}$.
\end{cor}
\begin{proof}
Let us write $G\equiv G(\beta)$. For each matrix monomial, we substitute
Lemma~\ref{lem:blkschur_ad} and note that the following holds for
all $j\ge0$
\begin{align*}
G^{j}=U\left[\begin{array}{c|c|c}
0 & G_{12} & G_{13}\\
\hline 0 & G_{22} & G_{23}\\
\hline 0 & 0 & 0
\end{array}\right]^{j+2}U^{-1} & =U\left[\begin{array}{c}
G_{12}\\
\hline G_{22}\\
\hline 0
\end{array}\right]G_{22}^{j}\left[\begin{array}{c|c|c}
0 & G_{22} & G_{23}\end{array}\right]U^{-1}.
\end{align*}
Repeating this argument for each monomial in $p(\cdot)$, we find
that
\begin{align*}
\|p(G)\,G^{2}\| & \le\kappa_{U}\|\begin{bmatrix}G_{12} & G_{22}\end{bmatrix}\|\|\begin{bmatrix}G_{22} & G_{23}\end{bmatrix}\|\|p(G_{22})\|,\\
 & \le\kappa_{U}(\beta\|D^{-1}A^{T}\|+1)(1+1)\|p(G_{22})\|,
\end{align*}
where $\kappa_{U}=\|U\|\|U^{-1}\|=\max\{\sigma_{\max}(B),1/\sigma_{\min}(B),\sigma_{\max}(B)/\sigma_{\min}(B)\}$. 
\end{proof}
Accordingly, the residual minimization problem (\ref{eq:gmres_poly2})
posed over $G(\beta)$ is reduced to a simpler problem over $K(\beta)$
after two iterations.
\begin{lemma}
\label{lem:admm_gmres}Fix $\beta>0$ and initial point $u^{0}$.
Let $u^{k}$ be the iterate generated at the $k$-th iteration of
ADMM-GMRES (Algorithm~\ref{alg:ADMM-GMRES}). Then the following
holds for all $k\ge2$ 
\begin{equation}
\frac{\|u^{k}-u^{\star}\|_{M}}{\|u^{0}-u^{\star}\|_{M}}\le(c_{0}+c_{1}\beta)\min_{\begin{subarray}{c}
p\in\P_{k-2}\\
p(1)=1
\end{subarray}}\|p(K(\beta))\|,\label{eq:minres_prob}
\end{equation}
where $K(\beta)$ is defined in (\ref{eq:Kdef}), $\|\cdot\|_{M}$
is defined in (\ref{eq:gmres_cost}), $\P_{k}$ is the space of order-$k$
polynomials, and $c_{0},c_{1}$ are constants.
\end{lemma}
\begin{proof}
Substituting $A\gets(I-G(\beta))$ and $b\gets r$ into (\ref{eq:gmres_poly2})
yields 
\begin{multline*}
\min_{\begin{subarray}{c}
p\in\P_{k}\\
p(0)=1
\end{subarray}}\|p(I-G(\beta))r\|\overset{\text{(a)}}{=}\min_{\begin{subarray}{c}
p\in\P_{k}\\
p(1)=1
\end{subarray}}\|p(G(\beta))r\|\\
\overset{\text{(b)}}{\le}\|r\|c(\beta)\min_{\begin{subarray}{c}
p\in\P_{k-2}\\
p(1)=1
\end{subarray}}\|p(G_{22}(\beta))\|\overset{\text{(c)}}{=}\|r\|c(\beta)\min_{\begin{subarray}{c}
p\in\P_{k-2}\\
p(1)=1
\end{subarray}}\|p(K(\beta))\|.
\end{multline*}
Equality (a) shifts the polynomials $p(1-z)\leftrightarrow p(z')$,
which also shifts the constraint point from $z=0$ to $z'=1$. Inequality
(b) takes the heuristic choice of $p(z)=z^{2}q(z)$ with an order
$k-2$ polynomial $q$, and substitutes Corollary~\ref{cor:poly_2norm}.
Equality (c) then shifts and scales the polynomials $p(z)\leftrightarrow p(\frac{1}{2}+\frac{1}{2}z')$,
keeping the constraint point $z=1$ at $z'=1$.
\end{proof}

\section{\label{sec:Worst-Case-Behavior}Worst-case behavior}

When applied to (\ref{eq:ecqp}), ADMM converges at the rate of $1-1/\sqrt{\kappa}$
with the parameter choice of $\beta=\sqrt{\mu L}$; a number of previous
authors have established versions of the following statement~\cite{deng2012global,giselsson2014diagonal,ghadimi2015optimal,nishihara2015general}.
\begin{proposition}
\label{prop:admm_spec_rad}The $k$-th iterate of ADMM with $k\ge2$
satisfies
\[
\frac{\|u^{k}-u^{\star}\|_{M}}{\|u^{0}-u^{\star}\|_{M}}\le(c_{0}+c_{1}\beta)\left(\frac{\gamma}{\gamma+1}\right)^{k-2}
\]
where $\gamma=\max\{L/\beta,\beta/\mu\}$, and $c_{0},c_{1}$ are
constants. The bound is sharp up to a multiplicative constant.
\end{proposition}
\begin{proof}
The residuals satisfy $r^{k}=G^{k}(\beta)r^{0}$, so $\|r^{k}\|/\|r^{0}\|\le\|G^{k}(\beta)\|$.
To establish the inequality, we substitute the bound $\|G_{22}^{k}\|\le(\frac{1}{2}+\frac{1}{2}\|K\|)^{k}$
from Lemma~\ref{lem:blkschur_ad} and the value of $\|K\|$ from
Remark~\ref{rem:Knrm} into Corollary~\ref{cor:poly_2norm}. To
prove sharpness, we take $n=\ell=2m$ and set $A=I_{n}$, $B=[I_{m},0_{m}]^{T}$,
and $D$ diagonal. Then, both $G_{22}$ and $K$ are diagonal by construction,
so $\|G_{22}^{k}\|=(\frac{1}{2}+\frac{1}{2}\|K\|)^{k}$ trivially
holds.
\end{proof}
Let us use Lemma~\ref{lem:admm_gmres} to prove a similar statement
for ADMM-GMRES.
\begin{theorem}
\label{thm:worst-case}The $k$-th iteration of ADMM-GMRES satisfies
\[
\frac{\|u^{k}-u^{\star}\|_{M}}{\|u^{0}-u^{\star}\|_{M}}\le\left(c_{0}+c_{1}\beta\right)\left(\frac{\gamma-1}{\gamma+1}\right)^{k-2}
\]
where $\gamma=\max\{L/\beta,\beta/\mu\}$, and $c_{0},c_{1}$ are
constants. The bound is sharp up to a multiplicative constant.
\end{theorem}
\begin{proof}
To establish the inequality, we set $p(z)$ in Lemma~\ref{lem:admm_gmres}
to be the monomial $p(z)=z^{k-2}$ and take $\|K^{k}\|\le\|K\|^{k}$.
To prove sharpness, we take $n=\ell=2m$ and give a problem construction
satisfying $\|K^{k}\|=\|K\|^{k}$ whose optimal polynomial is precisely
$p^{\star}(z)=z^{k-2}$. Consider
\begin{gather*}
A=I_{n},\qquad D=\begin{bmatrix}\frac{1}{\sqrt{\kappa}}I_{m} & 0\\
0 & \sqrt{\kappa}I_{m}
\end{bmatrix},\\
B=\begin{bmatrix}\cos\Theta\\
\sin\Theta
\end{bmatrix}\text{ where }\Theta=\frac{\pi}{2n}\diag(1,3,5,\ldots,n-1).
\end{gather*}
By inspection, $\mu=1/\sqrt{\kappa}$, $L=\sqrt{\kappa}$, and $L/\mu=\kappa$,
and $K(\sqrt{\mu L})$ is a scaled orthogonal matrix
\[
K(\sqrt{\mu L})=\frac{\sqrt{\kappa}-1}{\sqrt{\kappa}+1}\begin{bmatrix}\cos2\Theta & -\sin2\Theta\\
\sin2\Theta & \cos2\Theta
\end{bmatrix}
\]
whose $n$ eigenvalues lie evenly spaced along the circumference of
a circle centered at the origin with radius $a=\|K(\sqrt{\mu L})\|=\frac{\sqrt{\kappa}-1}{\sqrt{\kappa}+1}$.
The associate eigenvalue approximation problem is bound:
\[
\min_{\begin{subarray}{c}
p\in\P_{k-2}\\
p(1)=1
\end{subarray}}\|p(K(\sqrt{\mu L}))\|\overset{\text{(a)}}{=}\min_{\begin{subarray}{c}
p\in\P_{k-2}\\
p(1)=1
\end{subarray}}\max_{i\in\{1,\ldots,n\}}|p(a\omega^{i})|\overset{\text{(b)}}{\le}\min_{\begin{subarray}{c}
p\in\P_{k-2}\\
p(1)=1
\end{subarray}}\max_{\begin{subarray}{c}
\lambda\in\C\\
|\lambda|=a
\end{subarray}}|p(\lambda)|\overset{\text{(c)}}{=}a^{k-2},
\]
where $\omega=\exp(\sqrt{-1}\frac{2\pi}{n})$ is the $n$-th root
of unity. Step (a) makes a unitary eigendecomposition for the normal
matrix $K(\sqrt{\mu L})=XSX^{*}$, where $XX^{*}=X^{*}X=I$, and notes
that $\|p(K(\sqrt{\mu L}))\|=\|Xp(S)X^{*}\|=\|p(S)\|$. Step (b) encompasses
the roots of unity $\omega^{i}$ within the unit circle $\{z\in\C:|z|=1\}$.
Step (c) applies the closed-form solution $p^{\star}(z)=z^{k-2}$
due to Zarantonello (see~\cite{rivlin1974chebyshev} for a proof,
or~\cite{driscoll1998potential} for a more intuitive explanation).
In the limit $n\to\infty,$ the roots of unity converge uniformly
to the unit circle, and the inequality (b) converge uniformly towards
an equality. 
\end{proof}
\begin{figure}
\hfill{}\includegraphics[width=0.5\columnwidth]{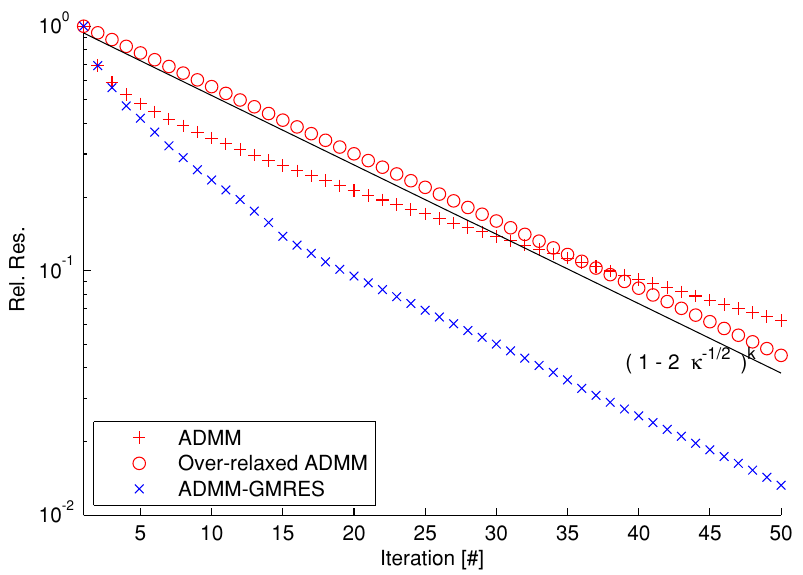}\hfill{}

\caption{\label{fig:worst-case}The problem construction in the proof of Theorem~\ref{thm:worst-case}
places the eigenvalues of $K(\beta)$ in a circle, so ADMM-GMRES convergences
at the same asymptotic rate as over-relaxed ADMM with $\omega=2$.}
\end{figure}
Setting $\beta=\sqrt{\mu L}$ minimizes value of $\gamma=\max\{L/\beta,\beta/\mu\}$.
This parameter choice allows ADMM to converge to an $\epsilon$-accurate
solution in 
\begin{equation}
\left\lceil (\sqrt{\kappa}+1)\log\left(\frac{c_{0}+c_{1}\beta}{\epsilon}\right)\right\rceil \text{ iterations,}\label{eq:admm_iter_bound}
\end{equation}
and ADMM-GMRES to do the same in
\begin{equation}
\left\lceil \frac{1}{2}(\sqrt{\kappa}+1)\log\left(\frac{c_{0}+c_{1}\beta}{\epsilon}\right)\right\rceil \text{ iterations.}\label{eq:orad_iter_bound}
\end{equation}
We see that ADMM-GMRES is only a factor of two better than basic ADMM.
In the worst case, GMRES will not be able to yield a substantial acceleration
over the basic ADMM method. This is easily verified numerically; see
Fig.~\ref{fig:worst-case}.

Indeed, the optimal polynomial used to prove Theorem~\ref{thm:worst-case}
may be extracted and explicitly applied as a successive over-relaxation
scheme. 
\begin{cor}
\label{cor:orad}Consider the successive over-relaxation (SOR) iterations
\begin{align*}
u^{1} & =G(\beta)u^{0}+v(\beta),\qquad u^{2}=G(\beta)u^{1}+v(\beta),\\
u^{j+1} & =(1-\omega)u^{j}+\omega[G(\beta)u^{j}+v(\beta)]\quad\forall j\in\{3,\ldots,k\}
\end{align*}
with $\beta=\sqrt{\mu L}$ and $\omega=2$. Then the $k$-th iterate
$u^{k}$ satisfies the bound in Theorem~\ref{thm:worst-case}.
\end{cor}
\begin{proof}
The SOR residuals satisfy $r^{k}=q(G)r^{0}$ where $q(z)=\prod_{i=1}^{k}[(1-\omega_{j})+\omega_{j}z]$.
Collocating its roots with those of the optimal polynomial $p^{\star}(z)$
in the proof of Theorem~\ref{thm:worst-case} yields the desired
iterates.
\end{proof}
This is precisely over-relaxed ADMM using the parameter choice of
$\omega=2$, which was shown to be optimal by several previous authors~\cite{davis2014faster,ghadimi2015optimal,nishihara2015general}.

\section{\label{sec:k4conv}Explaining Convergence in $O(\kappa^{1/4})$ Iterations}

In order to understand the circumstances that allow the ADMM-GMRES
converge an order-of-magnitude faster than basic ADMM, we make the
following assumption.
\begin{assume}[$\kappa_{X}$ is bounded]
\label{ass:kappaX}For a fixed $\beta>0$, the matrix $K(\beta)$,
defined in (\ref{eq:Kdef}), is diagonalizable. Furthermore, it has
an eigendecomposition $K(\beta)=XSX^{-1}$ whose matrix-of-eigenvectors
$X$ has a bounded condition number $\kappa_{X}=\|X\|\|X^{-1}\|$.
\end{assume}
Intuitively, we assume that the matrix $K(\beta)$ is close to normal,
so that its behavior can be accurately described by its eigenvalues
alone; we will return to this point later in Section~\ref{sec:normality}.
Substituting $\|p(K)\|=\|Xp(S)X^{-1}\|\le\|X\|\|p(S)\|\|X^{-1}\|$
reduces the residual minimization problem in Lemma~\ref{lem:admm_gmres}
to an \emph{eigenvalue approximation problem} (see e.g.~\cite{saad1986GMRES})
\begin{equation}
\frac{\|u^{k}-u^{\star}\|_{M}}{\|u^{0}-u^{\star}\|_{M}}\le(c_{0}+c_{1}\beta)\kappa_{X}\min_{\begin{subarray}{c}
p\in\P_{k-2}\\
p(1)=1
\end{subarray}}\|p(z)\|_{\Lambda},\label{eq:eig_approx}
\end{equation}
where we have used the maximum modulus notation
\[
\|p(z)\|_{\Lambda}\triangleq\max_{z\in\Lambda}|p(z)|,\qquad\Lambda\triangleq\{\lambda_{1},\ldots,\lambda_{m}\}.
\]
In this new problem, our objective is to construct a low-order polynomial
whose zeros are approximately the eigenvalues of $K(\beta)$. 

Problem (\ref{eq:eig_approx}) is made easier by enclosing the eigenvalues
$\Lambda$ within the disk $\mathcal{D}=\{z\in\C:|z|\le a\}$ mentioned
earlier in Remark~\ref{rem:Knrm}, with radius
\begin{equation}
a\triangleq\frac{\gamma-1}{\gamma+1}\text{ where }\gamma=\max\left\{ \frac{L}{\beta},\frac{\beta}{\mu}\right\} .\label{eq:disk_def}
\end{equation}
In view of Theorem~\ref{thm:worst-case}, this enclosure $\Lambda\subset\mathcal{D}$
is sharp: there exists a choice of problem data $A,B,D$ to place
$\Lambda$ right along the boundary $\partial\mathcal{D}$. The associated
optimal polynomial is simply $p^{\star}(z)=z^{k}$, but this causes
ADMM-GMRES to converge at the same rate as regular ADMM.

In order to improve upon the $O(\sqrt{\kappa})$ iteration estimate
from Theorem~\ref{thm:worst-case}, we must introduce additional
information about the distribution of eigenvalues within the interior
of the disk. Suppose, in particular, that all of our eigenvalues were
also real, i.e. $\Lambda\subset\R$. Then the corresponding approximation
problem over the real interval $\mathcal{I}\triangleq\R\cap\mathcal{D}$
has a closed-form solution attributed to Chebyshev (see~\cite[Ch.3]{greenbaum1997iterative}
or~\cite[Sec.6.11.1]{saad2003iterative})
\begin{equation}
\min_{\begin{subarray}{c}
p\in\P_{k}\\
p(1)=1
\end{subarray}}\|p(z)\|_{\mathcal{I}}=\frac{1}{|T_{k}(1/a)|}\le2\left(\frac{\sqrt{\gamma}-1}{\sqrt{\gamma}+1}\right)^{k},\label{eq:cheby}
\end{equation}
attained by $p^{\star}(z)=T_{k}(z/a)/|T_{k}(1/a)|$ where $T_{k}(z)$
is the order-$k$ Chebyshev polynomial of the first kind. Using the
Chebyshev polynomial to solve the eigenvalue approximation problem
(\ref{eq:eig_approx}) yields an optimal convergence rate of $(\kappa^{1/4}-1)/(\kappa^{1/4}+1)$
for the parameter choice $\beta=\sqrt{\mu L}$. In other words, ADMM-GMRES
converges to an $\epsilon$-accurate solution in $O(\kappa^{1/4})$
iterations, for an order of magnitude improvement over its worst-case.

\subsection{\label{subsec:cond}Damping the outlier eigenvalues}

\begin{figure}
\hfill{}\subfloat[]{\includegraphics[width=0.35\columnwidth]{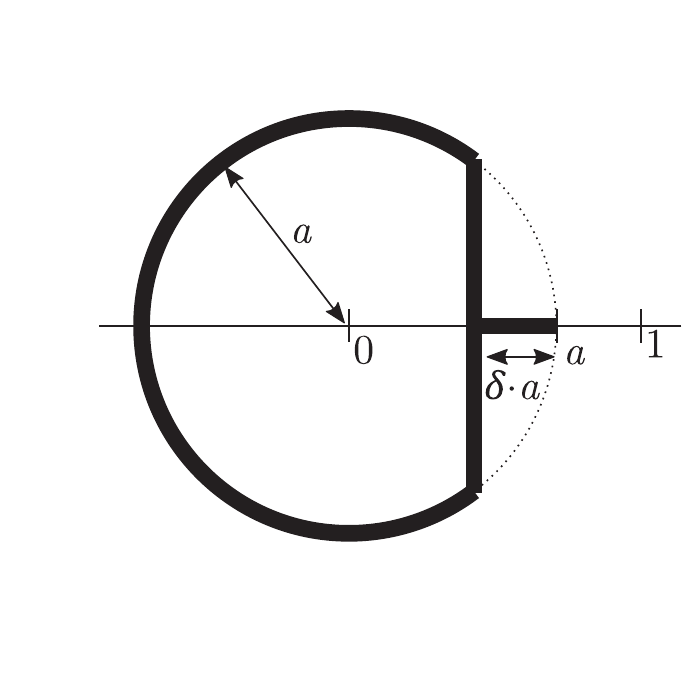}

}\hfill{}\subfloat[]{\includegraphics[width=0.35\columnwidth]{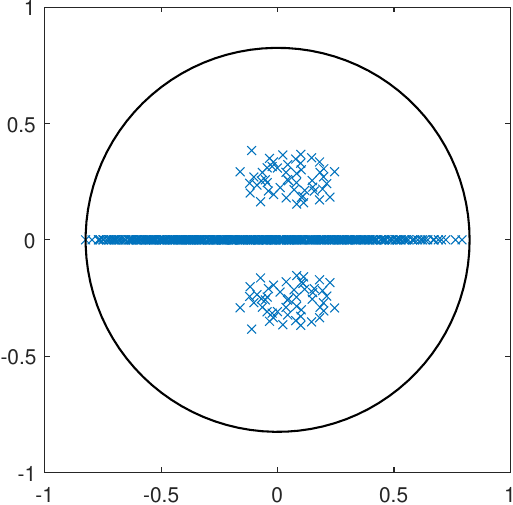}

}\hfill{}

\caption{\label{fig:disk-int}The enclosure $\mathcal{I}\cup\mathcal{C}$ bounds
the outlier eigenvalues away from the right-side of the disk by the
distance $\delta\cdot a$: (a) graphical illustration; (b) numerical
example, for a problem with $n=500$, $m=400$, $\ell=50$.}
\end{figure}
In practice, $K(\beta)$ also has a number of eigenvalues with nonzero
imaginary parts. These eigenvalues prevent (\ref{eq:cheby}) from
being directly applicable, so we refer to them as \emph{outlier} eigenvalues.
The issue of outliers is standard in Krylov subspace methods. If the
number of outliers is small, then a standard technique is to annihilate
them one at a time, and to apply a Chebyshev approximation to the
remaining eigenvalues that lie along a line; see~\cite[p.53]{greenbaum1997iterative}
or~\cite[Sec.5]{driscoll1998potential}. In our numerical experiments,
however, the number of outlier eigenvalues was often observed to be
quite large.

Instead, let us assume a different structure: that the outlier eigenvalues
are \emph{better conditioned} than the real eigenvalues. Rather than
annihilating them one at a time, it may be sufficient to ``dampen''
their effect using a few fixed-point iterations, like in multigrid
methods. Then, the Chebyshev approximation can be used to approximate
the remaining purely-real but poorly-conditioned eigenvalues. Since
GMRES is optimal, it must converge faster than this heuristic approach.

Consider the eigenvalue enclosure $\Lambda\{K\}\subset\mathcal{I}\cup\mathcal{C}$,
where $\mathcal{I}=\R\cap\mathcal{D}$ is the same real interval considered
in the previous section, and 
\begin{align}
\mathcal{C} & \triangleq\{z\in\mathcal{D}:\mathrm{Re}\{z\}\le\delta\cdot a\},\qquad\delta\triangleq1-\max_{\begin{subarray}{c}
\lambda\in\Lambda\{K(\beta)\}\\
\Imag\lambda\ne0
\end{subarray}}\frac{\mathrm{Re}\,\lambda}{\|K(\beta)\|}\ge0\label{eq:S_I_def}
\end{align}
is used to encompasses the outlier eigenvalues; an illustration is
shown in Fig.~\ref{fig:disk-int}. We view the quantity $\delta^{-1}$
as a \emph{relative} condition number of $\mathcal{C}$, due to the
following result.
\begin{lemma}
\label{lem:disk-minus-seg}The approximation problem for $\mathcal{C}$
in (\ref{eq:S_I_def}) is bounded
\[
\min_{\begin{subarray}{c}
p\in\P_{k}\\
p(1)=1
\end{subarray}}\|p(z)\|_{\mathcal{C}}\le\left(\frac{2a}{1+a}\right)^{k}\left(1-\frac{\delta}{2}\right)^{k/2},
\]
using the polynomial $p(z)=(z+a)^{k}/(1+a)^{k}$.
\end{lemma}
\begin{proof}
We use the over-relaxation polynomial $p(z)=(z+\omega)^{k}(1+\omega)^{-k}$
to approximate $\mathcal{C}$. Since $\omega>0$, the maximum modulus
is attained at $z^{\star}=\arg\max_{z\in\mathcal{C}}|p(z)|=a\left[(1-\delta)\pm j\sqrt{1-(1-\delta)^{2}}\right]$.
Now, 
\[
\frac{|z^{\star}+\omega|^{2}}{(1+\omega)^{2}}=\frac{[a(1-\delta)+\omega]^{2}+a^{2}[1-(1-\delta)^{2}]}{(1+\omega)^{2}}=\frac{(a+\omega)^{2}}{(1+\omega)^{2}}\left(1-\frac{2a\delta\omega}{(a+\omega)^{2}}\right).
\]
Setting $\omega=a$ maximizes the ratio $2\delta\omega/(a+\omega)^{2}$
and yields the desired bound.
\end{proof}
Fixing $\delta>0$, the error over $\mathcal{C}$ can be dampened
to some prescribed accuracy in a fixed number of iterations, independent
of all other considerations. This is the central insight that we use
in our new iteration estimate; so long as $\delta>0$, ADMM-GMRES
will converge in $O(\kappa^{1/4})$ iterations.
\begin{theorem}
\label{thm:disk_seg_conv}Under Assumption~\ref{ass:kappaX}, the
$k$-th iterate of ADMM-GMRES satisfies
\begin{equation}
\frac{\|u^{k}-u^{\star}\|_{M}}{\|u^{0}-u^{\star}\|_{M}}\le2(c_{0}+\beta c_{1})\kappa_{X}\left(\frac{\sqrt{\gamma}-1}{\sqrt{\gamma}+1}\right)^{\delta(k-2)/6},\label{eq:disk_seg_optim_constr}
\end{equation}
where $\delta$ is defined in (\ref{eq:S_I_def}), and $c_{0},c_{1},\kappa_{X}$
are constants.
\end{theorem}
Our proof solves the approximation problem (\ref{eq:eig_approx})
over $\mathcal{I}\cup\mathcal{C}\supset\Lambda$ using the following
polynomial
\[
p_{k}(z)\triangleq\left(\frac{z+a}{1+a}\right)^{\eta}\frac{T_{\xi}(z/a)}{|T_{\xi}(1/a)|},\qquad\xi\triangleq k-\eta,
\]
which is constructed as the product of $\eta$ fixed-point iterations
and an order-$\xi$ Chebyshev polynomial from (\ref{eq:cheby}). Intuitively,
each increment of $\xi$ decreases global error by $(\sqrt{\gamma}-1)/(\sqrt{\gamma}+1)$,
but also increases the relative error between $\mathcal{C}$ and $\mathcal{I}$
by a constant factor. To smooth this error between the two regions,
we increment $\eta$ by a fixed number determined by Lemma~\ref{lem:disk-minus-seg}.
Alternating between decrementing the global error and smoothing the
relative error allows us to converge at the overall accelerated rate
of $(\sqrt{\gamma}-1)/(\sqrt{\gamma}+1)$.
\begin{proof}
Noting that $\|p_{k}(z)\|_{\Lambda}\le\|p_{k}(z)\|_{\mathcal{I}\cup\mathcal{C}}=\max\left\{ \|p_{k}(z)\|_{\mathcal{I}},\|p_{k}(z)\|_{\mathcal{C}}\right\} $,
we bound each component
\begin{align}
\|p_{k}(z)\|_{\mathcal{C}} & \le\left(\frac{2a}{1+a}\right)^{\eta}\left(1-\frac{\delta}{2}\right)^{\eta/2}\frac{(1+\sqrt{2})^{\xi}}{T_{\xi}(1/a)},\label{eq:max_D}\\
\|p_{k}(z)\|_{\mathcal{I}} & =\left(\frac{2a}{1+a}\right)^{\eta}\frac{1}{T_{\xi}(1/a)}\le2\left(\frac{\sqrt{\gamma}-1}{\sqrt{\gamma}+1}\right)^{\xi},\label{eq:max_S}
\end{align}
using Lemma~\ref{lem:delta_bnd} and $\max_{|z|\le1}|T_{n}(z)|\le(1+\sqrt{2})^{n}$.
We will pick the ratio $\eta/\xi$ to satisfy
\begin{equation}
\left(1-\frac{\delta}{2}\right)^{\eta/\xi}\le\frac{1}{(1+\sqrt{2})^{2}},\label{eq:c_ineq}
\end{equation}
so that we have $\|p_{k}(z)\|_{\mathcal{I}}\ge\|p_{k}(z)\|_{\mathcal{C}}$.
Viewing $\eta/\xi$ as an ``iteration estimate'' to guarantee a
constant error reduction of $\epsilon=1/(1+\sqrt{2})^{2}$ over $\mathcal{C}$,
we take logarithms and obtain $\eta/\xi\ge2\delta^{-1}\log\epsilon^{-1}=c_{3}\delta^{-1}$
and $\xi\ge\delta(k-2)/6$, noting that $\delta\le2$. With $\|p_{k}(z)\|_{\mathcal{I}}=\|p_{k}(z)\|_{\mathcal{I}\cup\mathcal{C}}$
now guaranteed, we take the second expression in (\ref{eq:max_S})
to be the global error estimate $\|p_{k}(z)\|_{\mathcal{I}}\ge\|p_{k}(z)\|_{\Lambda}$
for (\ref{eq:eig_approx}).
\end{proof}
Theorem~\ref{thm:disk_seg_conv} says that ADMM-GMRES will converge
to an $\epsilon$-accurate solution in 
\[
2+\left\lceil \frac{6(\kappa^{1/4}+1)}{\delta}\log\left(\frac{2\kappa_{X}(c_{0}+c_{1}\beta)}{\epsilon}\right)\right\rceil \text{ iterations}
\]
using the parameter choice $\beta=\sqrt{\mu L}$. This is a factor
of $O((\kappa^{1/4}/\delta)\log(\kappa_{X}/\epsilon))$. So long as
$\delta$ is not too small relative to $1/\kappa^{1/4}$ and $\kappa_{X}$
not too big relative to $1/\epsilon$, Theorem~\ref{thm:disk_seg_conv}
guarantees convergence in $O(\kappa^{1/4})$ iterations.

\subsection{\label{subsec:explain_empirical}Explaining the empirical results}

\begin{figure}
\hfill{}\subfloat[]{\includegraphics[width=0.4\columnwidth]{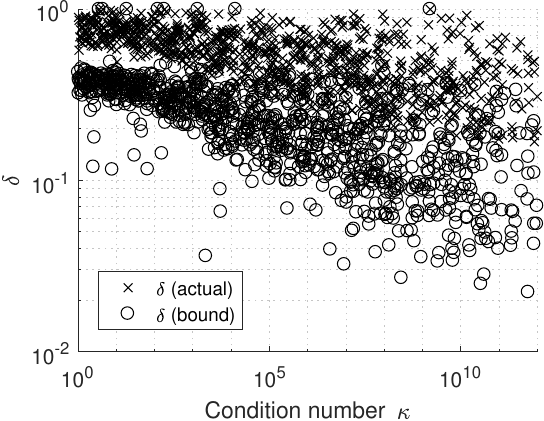}

}\hfill{}\subfloat[]{\includegraphics[width=0.4\columnwidth]{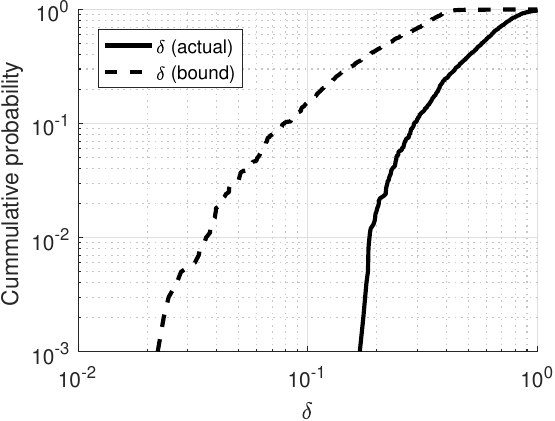}

}\hfill{}

\hfill{}\subfloat[]{\includegraphics[width=0.4\columnwidth]{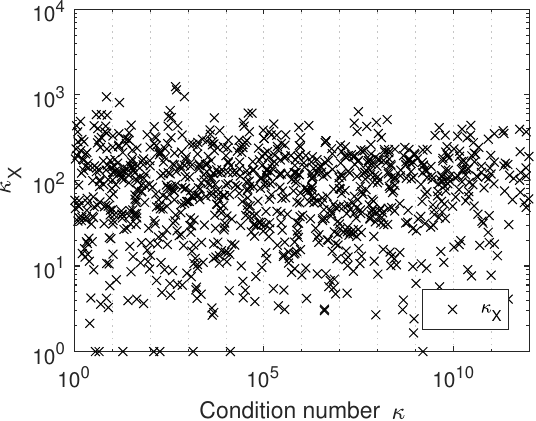}

}\hfill{}\subfloat[]{\includegraphics[width=0.4\columnwidth]{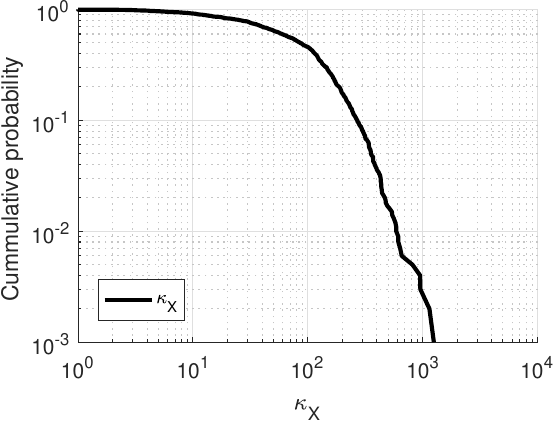}

}\hfill{}

\caption{\label{fig:thm_stats}Statistics for $\delta$ and $\kappa_{X}$ for
the 1000 randomly-generated problems in Fig.~\ref{fig:first_comparison}:
(a) \& (b) scatter plot and empirical CDF for $\delta$ and its lower
bound $\delta_{\protect\lb}$ (Lemma~\ref{lem:delta_bnd}); (c) \&
(d) scatter plot and empirical CDF for $\kappa_{X}$.}
\end{figure}

Earlier in the introduction, we presented a comparison of ADMM and
ADMM-GMRES for 1000 random trials. These problems were generated using
the following algorithm.

\begin{algor}\label{algr:log-normal}Input: dimension parameters
$n$, $\ell\le n$, $m\le\ell$ and conditioning parameter $s>0$.
\\
Output: random data matrices $D\in\S^{n}$, $A\in\R^{\ell\times n}$,
$B\in\R^{\ell\times m}$ satisfying Assumption~\ref{ass:regularity}.
\begin{enumerate}
\item Select the singular vectors $U_{A},U_{B},U_{D},V_{A},V_{B}$ i.i.d.
uniformly from their respective orthogonal groups. 
\item Select the singular values $\Sigma_{A},\Sigma_{B},\Sigma_{D}$ i.i.d.
from the log-normal distribution $\sim\mathrm{exp}(0,s^{2})$. 
\item Output $A=U_{A}\Sigma_{A}V_{A}^{T}$, $B=U_{B}\Sigma_{B}V_{B}^{T}$,
and $D=U_{D}\Sigma_{D}U_{D}^{T}$.
\end{enumerate}
\end{algor}

More specifically, the dimension parameters $n,\ell,m$ were uniformly
sampled from $n=1000$, $\ell\in\{1,\ldots,n\}$, and $m\in\{1,\ldots,\ell\}$,
and the log-standard-deviation is swept within the range $s\in[0,2]$.
Both algorithms are tasked with solving the equation to a relative
residual of $\epsilon=10^{-6}$. ADMM fails to converge within 100,000
iterations for 49 of the problems, while ADMM-GMRES converges on all
of the problems.

To verify whether Theorem~\ref{thm:disk_seg_conv} is sufficient
to explain the $O(\kappa^{1/4})$ behavior seen in these 1000 random
problems, we plot the distribution of $\delta$ and $\kappa_{X}$
with respect to $\kappa$ in Fig.~\ref{fig:thm_stats}a and Fig.~\ref{fig:thm_stats}c.
The smallest value of $\delta$ is 0.06, with mean and median both
around 0.6. The largest value of $\kappa_{X}$ is 775, with mean and
median both around 50. These are both relatively modest, and as predicted
by Theorem~\ref{thm:disk_seg_conv}, ADMM-GMRES converges in $O(\kappa^{1/4})$
iterations.

The associated cumulative probability distributions are shown in Fig.~\ref{fig:thm_stats}b
and Fig.~\ref{fig:thm_stats}d. An exponentially decaying probability
tail for both quantities can be observed. The rapid roll-off in probability
tail is a signature trait for \emph{concentration-of-measure} type
results. In the case of $\delta$, consider the following bound.
\begin{lemma}
\label{lem:delta_bnd}Define $K(\beta)$, $Q$, $P$ as in (\ref{eq:Kdef}).
Then
\[
\delta_{\lb}\triangleq1-\frac{\|Q^{T}K(\beta)Q\|+\|P^{T}K(\beta)P\|}{2\|K(\beta)\|}\le\delta.
\]
\end{lemma}
\begin{proof}
Note that $K(\beta)$ has the following block structure $K=\left[\begin{smallmatrix}X & Z\\
-Z^{T} & Y
\end{smallmatrix}\right]$. For such matrices, Benzi~\&~Simoncini~\cite{benzi2006eigenvalues}
used a field-of-values type argument to show that if $\lambda\in\Lambda\{K\}$
and $\mathrm{Im}\{\lambda\}\ne0$, then $|\mathrm{Re}\lambda|\le\frac{1}{2}\left[\|X\|+\|Y\|\right]$.
Substituting the definitions of $\delta$, $X$, and $Y$ results
in the desired bound.
\end{proof}
Hence, we see that the quantity $\delta$ is bounded away from zero
because the matrices $K$, $Q$, $P$ are incoherent. More specifically,
let us write $K=U\Sigma V^{T}$ as its singular value decomposition.
If we treat $Q$, $P$, $U$, and $V$ all as random orthogonal matrices,
then the matrices $Q^{T}U$, $P^{T}U$, $Q^{T}V$, and $P^{T}V$ are
all dense with an overwhelming probability~\cite[Thm.VIII.1]{donoho2001uncertainty}.
This observation bounds the expected value of $\|Q^{T}K(\beta)Q\|$
and $\|P^{T}K(\beta)P\|$ away from $\|K(\beta)\|$, thereby bounding
$\delta$ away from zero via Lemma~\ref{lem:delta_bnd}.

\section{\label{sec:normality}The normality assumption}

A weakness in our argument is Assumption~\ref{ass:kappaX}, which
takes $\kappa_{X}$, the condition number for the matrix of eigenvectors
of $K(\beta)$, to be bounded. The assumption is closely related to
the \emph{normality} of $K(\beta)$. A matrix is normal if it has
a complete set of orthogonal eigenvectors, so if $K(\beta)$ is normal,
then $\kappa_{X}=1$, and our bounds are sharp up to a multiplicative
factor. On the other hand, if $K(\beta)$ is nonnormal, then our bounds
may fail to be sharp to an arbitrary degree. The phenomenon has to
do with the fact that eigenvalues are not necessarily meaningful descriptors
for the behavior of nonnormal matrices; see the discussions in~\cite{greenbaum1997iterative,driscoll1998potential,embree1999descriptive}
for more details, and the book~\cite{trefethen2005spectra} for a
thorough exposition.

In the case of ADMM, there are reasons to believe that $K(\beta)$
is relatively close to normal, and that Assumption~\ref{ass:kappaX}
is not too strong in practice. To explain, consider the following
dimensionless nonnormality measure
\[
\nu(A)\triangleq\|A^{T}A-AA^{T}\|_{F}^{1/2}/\|A\|_{F},
\]
which takes on values from 0 (attained by any normal matrix) to $\sqrt{2}$
(attained by highly nonnormal matrices like $\left[\begin{smallmatrix}0 & 0\\
1 & 0
\end{smallmatrix}\right]$). The measure is closely associated to Henrici's departure from normality~\cite{henrici1962bounds},
and can be used to bound many other measures of nonnormality; see
the survey in~\cite{elsner1987measures}.
\begin{proposition}
\label{prop:normality}Let $K(\beta)$ be the $\ell\times\ell$ matrix
in (\ref{eq:Kdef}). Then
\[
\nu(K(\beta))\le\frac{\left(8\min\{m,\ell-m\}\right)^{1/4}}{\|K(\beta)\|_{F}/\|K(\beta)\|}\le\sqrt{2}\ell^{1/4}\frac{\|K(\beta)\|}{\|K(\beta)\|_{F}}.
\]
\end{proposition}
\begin{proof}
Note that $K\equiv K(\beta)$ has the structure $K=JU^{T}WU$, where
$J=\diag(I_{m},-I_{\ell-m})$, $U$ is orthonormal, and $W=W^{T}$
shares its singular values with $K$. Then $\|K^{T}K-KK^{T}\|_{F}^{2}=\|UW^{2}U-JU^{T}W^{2}UJ\|_{F}^{2}=2\|2Q^{T}W^{2}P\|_{F}^{2}$,
where $Q$ is the first $m$ columns of $U$, and $P$ is the remaining
$\ell-m$ columns. But $\|Q^{T}W^{2}P\|_{F}^{2}=\tr W^{2}QQ^{T}W^{2}PP^{T}\le\|W\|^{4}\min\{\tr QQ^{T},\tr PP^{T}\}$,
and taking fourth roots produces the first inequality. The second
inequality follows by maximizing the bound with $\ell=m/2$.
\end{proof}
In the literature, the ratio $\|K\|_{F}^{2}/\|K\|^{2}$ is sometimes
known as the \emph{numerical rank} of $K$; see~\cite{rudelson2007sampling}.
Taking on a value from $1$ to $\ell$, it is always bounded by, and
is a stable relaxation of, the rank of $K$.

Assuming that the numerical rank of $K$ grows linearly with its dimension
$\ell$ (e.g. if the data were generated using Algorithm~\ref{algr:log-normal}),
then substituting $\|K\|_{F}^{2}/\|K\|^{2}\in\Omega(\ell)$ into Proposition~\ref{prop:normality}
produces $\nu(K)\in O(\ell^{-1/4})$. The matrix $K$ becomes more
and more normal as its dimension $\ell$ grows large, since $\nu(K)$
decays to zero. While the observation does not provide a rigorous
bound for $\kappa_{X}$, it does concur with our numerical results
presented in later sections.

Finally, even if Assumption~\ref{ass:kappaX} fails to hold, ADMM-GMRES
will still obey its worst-case bound (Proposition~\ref{thm:worst-case}).
Convergence is guaranteed in $O(\sqrt{\kappa})$ iterations, although
the significant acceleration from GMRES may be lost.

\section{\label{sec:num_res}Comparison with classical preconditioners}

\global\long\def\A{\mathbf{A}}
\global\long\def\vector{\mathrm{vec}\,}
Throughout this paper, we have treated ADMM-GMRES as a preconditioned
Krylov subspace method for the KKT equations associated with (\ref{eq:ecqp})
\begin{equation}
\left[\begin{array}{cc|c}
D & 0 & A^{T}\\
0 & 0 & B^{T}\\
\hline A & B & 0
\end{array}\right]\left[\begin{array}{c}
x\\
z\\
\hline y
\end{array}\right]=\left[\begin{array}{c}
-c\\
-p\\
\hline d
\end{array}\right],\label{eq:saddle}
\end{equation}
while assuming that matrix-vector products with $(\beta^{-1}D+A^{T}A)^{-1}$,
$(B^{T}B)^{-1}$, $A$, $B$, $A^{T}$ and $B^{T}$ can be efficiently
performed, i.e. efficient oracles are available. But (\ref{eq:saddle})
is a standard saddle-point system\textemdash albeit with a singular
$(1,1)$ block\textemdash and preconditioned Krylov subspace methods
for such problems are mature and well-developed; we refer the reader
to the authoritative surveys~\cite{benzi2005numerical,axelsson2015unified}.
A number of classical preconditioners can be constructed using these
same oracles, many of them even sharing the same $O(\sqrt{\kappa})$
worst-case iteration bound as ADMM. 

An important finding in this paper is that the average-case behavior
of ADMM-GMRES is considerably better than its worst-case. In fact,
our results from Section~\ref{sec:k4conv} suggest that the worst-case
bound is almost never attained, except on artificially constructed
``degenerate'' problems. It is natural to ask whether the same thing
can be said for the alternative preconditioners, which are constructed
using the same ingredients. Will they also converge in $O(\kappa^{1/4})$
iterations? Or will they readily attain their worst-case bound of
$O(\sqrt{\kappa})$ iterations?

In this section, we benchmark ADMM-GMRES against classical preconditioned
Krylov subspace methods for saddle-point problems, on random instances
of (\ref{eq:saddle}) generated using Algorithm~\ref{algr:log-normal}.
We restrict our attention to classical preconditioners that are based
on the same six matrix-vector products listed above, and also matrix-vector
products with $AD^{-1}A^{T}$ and its inverse $(AD^{-1}A^{T})^{-1}$,
the former of which arises by block-eliminating $D$ from (\ref{eq:saddle}).
Our goal is to compare the \emph{number of iterations} needed by different
preconditioners to solve the same $\kappa$-conditioned problem to
$\epsilon$-accuracy. 

A practical issue overlooked by a direct comparison of iteration counts
(and the number of oracle calls) is that some oracles are considerably
more expensive to call. To account for this, we use the CPU time as
a \emph{weighted tally of oracle calls}, by implementing the \emph{relative
timings} of the oracles proportional to their real-life values. To
this end, we implement the matrix-vector products with $A,$ $B,$
$A^{T},$ and $B^{T}$ explicitly, and make the following assumption
to reduce the cost of matrix-vector products with $(\beta^{-1}D+A^{T}A)^{-1},$
$(B^{T}B)^{-1},$ $AD^{-1}A^{T},$ and $(AD^{-1}A^{T})^{-1}$.
\begin{assume}
\label{ass:precomp}The Cholesky factorizations $B^{T}B=L_{B}L_{B}^{T}$,
$D=L_{D}L_{D}^{T}$, and the eigendecomposition $(AD^{-1}A^{T})^{-1}=V\Lambda V^{T}$
are explicitly available. 
\end{assume}
Using Assumption~\ref{ass:precomp}, matrix-vector products with
$(B^{T}B)^{-1}$ may be implemented as $z\mapsto L_{B}^{-T}L_{B}^{-1}z$,
those with $AD^{-1}A^{T},$ $(AD^{-1}A^{T})^{-1}$ may be implemented
as $y\mapsto V\Lambda^{-1}V^{T}y$ and $y\mapsto V\Lambda V^{T}y$,
and those with $(\beta^{-1}D+A^{T}A)^{-1}$ may be implemented as
\begin{align*}
x\mapsto(\beta^{-1}D+A^{T}A)^{-1}x & =\beta D^{-1}\left[I-A^{T}(\beta^{-1}I+AD^{-1}A^{T})^{-1}AD^{-1}\right]x.
\end{align*}
The assumption is modeled after the Newton subproblem in Section 8,
which satisfies it in near-linear time due to a Kronecker structure
described in Section~\ref{subsec:Implementation}. It is slightly
stronger that what is strictly necessary to realize the oracles for
the comparison; we adopt it because it closely mimics the implementation
considerations for Section 8. 

\subsection{Classical preconditioners}

We restrict our attention to classical preconditioners that can be
realized using only matrix-vector products with $(\beta^{-1}D+A^{T}A)^{-1},$
$(B^{T}B)^{-1},$ $A,$ $B,$ $A^{T},$ $B^{T}$, $AD^{-1}A^{T},$
and $(AD^{-1}A^{T})^{-1}$. These methods are selected from the survey~\cite{benzi2005numerical},
and solve either the reduced augmented system
\begin{equation}
\left[\begin{array}{c|c}
0 & B^{T}\\
\hline B & -AD^{-1}A^{T}
\end{array}\right]\left[\begin{array}{c}
z\\
\hline y
\end{array}\right]=\left[\begin{array}{c}
-p\\
\hline \tilde{d}
\end{array}\right],\label{eq:kkt2}
\end{equation}
or the Schur complement problem
\begin{equation}
\left[B^{T}(AD^{-1}A^{T})^{-1}B\right]z=-\tilde{p},\label{eq:schur}
\end{equation}
where the new right-hand sides are obtained via forward substitution
\begin{equation}
\tilde{d}=d+AD^{-1}c,\qquad\tilde{p}=p-B^{T}(AD^{-1}A^{T})^{-1}\tilde{d},\label{eq:fwdsub}
\end{equation}
and the unknown variables are recovered via back substitution
\begin{equation}
y=(AD^{-1}A^{T})^{-1}(Bz-\tilde{d}),\qquad x=-D^{-1}(A^{T}y+c).\label{eq:bwdsub}
\end{equation}
Note that each forward substitution (\ref{eq:fwdsub}), backward substitution
(\ref{eq:bwdsub}), and matrix-vector product with (\ref{eq:kkt2})
and (\ref{eq:schur}) can be performed using only the 8 matrix-vector
oracles listed above.

\textbf{Block-diagonal preconditioner (Blk-Diag).} Solve the reduced
augmented system (\ref{eq:kkt2}) using a symmetric indefinite Krylov
method like MINRES, with the positive definite matrix
\[
M_{1}=\left[\begin{array}{c|c}
\beta B^{T}B & 0\\
\hline 0 & AD^{-1}A^{T}
\end{array}\right]
\]
serving as preconditioner. The $(2,2)$ block of $M_{1}$ matches
that of $\hat{H}$, while its $(1,1)$ block is used to precondition
the Schur complement. The preconditioned matrix has the eigenvalue
$-1$ with multiplicity $m-\ell$, and $2\ell$ eigenvalues $\lambda_{i}=\frac{1}{2}(-1\pm\sqrt{1+4\eta_{i}})$,
where $\eta_{1},\ldots,\eta_{\ell}$ are the $\ell$ eigenvalues of
$\beta^{-1}Q^{T}\tilde{D}Q$~\cite[Lem.2.1]{fischer1998minimum}
(see also~\cite[Thm.3.8]{benzi2005numerical} and \cite[Lem.2.1]{rusten1992preconditioned}).
Applying the classic two-interval approximation result~\cite{de1982extremal}
(see also~\cite[Ch.3]{greenbaum1997iterative}) shows that MINRES
converges to an $\epsilon$-accurate solution within $k\le1+\left\lceil \sqrt{\kappa}\log[(1+\sqrt{1+4L/\beta})/\epsilon]\right\rceil $
iterations. We set $\beta=L$ to obtain convergence in $O(\sqrt{\kappa})$
iterations.

\textbf{Constraint preconditioner I (Constr I). }Solve the reduced
augmented system (\ref{eq:kkt2}) using a general Krylov method like
GMRES, with
\[
M_{2}=\left[\begin{array}{c|c}
0 & B^{T}\\
\hline B & -\beta I
\end{array}\right]=\left[\begin{array}{c|c}
I & \beta^{-1}B^{T}\\
\hline 0 & I
\end{array}\right]\left[\begin{array}{c|c}
\beta^{-1}B^{T}B & 0\\
\hline 0 & -\beta I
\end{array}\right]\left[\begin{array}{c|c}
I & 0\\
\hline \beta^{-1}B & I
\end{array}\right]
\]
serving as preconditioner. The preconditioner is designed to replicate
the governing matrix, while modifying the $(2,2)$ block in a way
as to make the overall matrix considerably easier to invert. The preconditioned
matrix has the eigenvalue 1 with multiplicity $2\ell$, and $m-\ell$
eigenvalues that coincide with the eigenvalues of $\beta^{-1}P^{T}\tilde{D}^{-1}P$~\cite[Thm.2.1]{keller2000constraint};
see also~\cite[Thm.10.1]{benzi2005numerical}. The latter $m-\ell$
eigenvalues lie within the real interval $[1/(L\beta),1/(\mu\beta)]$,
so assuming diagonalizability (i.e. adopting a version of Assumption~\ref{ass:kappaX}),
GMRES converges within $O(\sqrt{\kappa})$ iterations for all choices
of $\beta\ge1/L$. We set $\beta=\sqrt{\mu L}$ to concur with ADMM.

\textbf{Constraint preconditioner II (Constr II). }Solve the Schur
complement system (\ref{eq:schur}) using a symmetric positive definite
Krylov method like conjugate residuals, with 
\[
M_{3}=B^{T}B
\]
serving as preconditioner. This is derived by using the Schur complement
from the previous preconditioner to precondition the Schur complement
of (\ref{eq:kkt2}). The preconditioned problem has coefficient matrix
$Q^{T}\tilde{D}Q,$ whose eigenvalues lie in the real interval $[\mu,L]$.
Accordingly, conjugate residuals converges within $O(\sqrt{\kappa})$
iterations.

\textbf{Hermitian Skew-Hermitian Splitting (HSS).} Solve the reduced
augmented system (\ref{eq:kkt2}) using a general Krylov method like
GMRES, with
\[
M_{4}=\left[\begin{array}{c|c}
\alpha I & 0\\
\hline 0 & -(AD^{-1}A^{T}+\alpha I)
\end{array}\right]\left[\begin{array}{c|c}
\alpha I & B^{T}\\
\hline -B & \alpha I
\end{array}\right]
\]
as preconditioner. Small choices of the parameter $\alpha$ work best,
though the method is not sensitive to its exact value~\cite{benzi2003optimization,benzi2004preconditioner,simoncini2004spectral}.
Note that $M_{4}$ requires matrix-vector products with $(\alpha^{2}I+B^{T}B)^{-1}$
to be efficient. When $\alpha$ is sufficiently small, this matrix
may be approximated, e.g. using a few iteration of conjugate gradients
preconditioned by $B^{T}B$. To keep our implementation simple, we
set $\alpha=1/L$ (as recommended by Simoncini and Benzi~\cite{simoncini2004spectral})
and explicitly precompute a Cholesky factorization for $\alpha^{2}I+B^{T}B$.

Our list excludes the Uzawa method (and its inexact variants), incomplete
factorizations, and multilevel / hierarchical preconditioners, because
they cannot be efficiently realized using the 8 matrix-vector oracles
alone. We have excluded the Arrow\textendash Hurwicz as it is simply
a lower-cost, less accurate version of ``Constr I''. We have also
excluded the block-triangular version of ``Blk-Diag'', because it
can be shown to be almost identical to ``Constr II'' after a single
iteration, but requires the more computationally expensive GMRES algorithm.

\subsection{Results}

\begin{table}
\caption{\label{tab:1k}Max. Iterations (Max. CPU time in seconds) to $\epsilon=10^{-6}$
for random problems with dimensions $n=1000$, $1\le m\le n$, $1\le\ell\le m$.
``ADGM'' refers to ADMM-GMRES, and ``ADGM($k$)'' refers to ADMM-GMRES($k$),
i.e. with restart parameter $k$.}

\hfill{}%
\begin{tabular}{c|c|c|c|c|c}
$\log_{10}\kappa$ & \multicolumn{1}{c|}{$(0,2]$} & \multicolumn{1}{c|}{$(2,4]$} & \multicolumn{1}{c|}{$(4,6]$} & $(6,8]$ & $(8,10]$\tabularnewline
Num. trials & 204 & 192 & 169 & 185 & 135\tabularnewline
\hline 
$10\sqrt{\kappa}$ & \multicolumn{1}{c|}{32} & \multicolumn{1}{c|}{316} & \multicolumn{1}{c|}{3162} & 31,623 & 316,228\tabularnewline
ADMM & 126 (0.97) & 982 (5.09) & $>10^{3}$ & $>10^{3}$ & $>10^{3}$\tabularnewline
Blk-Diag & 102 (0.52) & 506 (3.27) & $>10^{3}$ & $>10^{3}$ & $>10^{3}$\tabularnewline
Constr I & 42 (0.45) & 155 (1.19) & 269 (2.11) & 553 (3.10) & 678 (4.43)\tabularnewline
Constr II & 49 (0.19) & 227 (1.28) & 722 (3.00) & $>10^{3}$ & $>10^{3}$\tabularnewline
HSS & 97 (2.72) & 278 (12.6) & 532 (19.4) & $>10^{3}$ & $>10^{3}$\tabularnewline
\hline 
$6\kappa^{1/4}$ & \multicolumn{1}{c|}{11} & \multicolumn{1}{c|}{34} & \multicolumn{1}{c|}{107} & 337 & 1067\tabularnewline
ADGM & 13 (0.23) & 29 (0.55) & 76 (1.24) & 198 (3.94) & 469 (6.44)\tabularnewline
ADGM(5) & 14 (0.25) & 50 (0.68) & $>10^{3}$ & $>10^{3}$ & $>10^{3}$\tabularnewline
ADGM(10) & 13 (0.21) & 37 (0.51) & $>10^{3}$ & $>10^{3}$ & $>10^{3}$\tabularnewline
ADGM(25) & 13 (0.20) & 30 (0.50) & $>10^{3}$ & $>10^{3}$ & $>10^{3}$\tabularnewline
\hline 
\end{tabular}\hfill{}

\medskip{}

\caption{\label{tab:3k}Max. Iterations (Max. CPU time in seconds) to $\epsilon=10^{-6}$
for random problems with dimensions $n=3000$, $1\le m\le n$, $1\le\ell\le m$.}

\hfill{}%
\begin{tabular}{c|c|c|c|c|c}
$\log_{10}\kappa$ & \multicolumn{1}{c|}{$(0,2]$} & \multicolumn{1}{c|}{$(2,4]$} & \multicolumn{1}{c|}{$(4,6]$} & $(6,8]$ & $(8,10]$\tabularnewline
Num. trials & 242 & 236 & 210 & 195 & 110\tabularnewline
\hline 
$10\sqrt{\kappa}$ & \multicolumn{1}{c|}{32} & \multicolumn{1}{c|}{316} & \multicolumn{1}{c|}{3162} & 31,623 & 316,228\tabularnewline
ADMM & 85 (9.87) & 910 (98.8) & $>10^{3}$ & $>10^{3}$ & $>10^{3}$\tabularnewline
Blk-Diag & 96 (4.14) & 568 (20.2) & $>10^{3}$ & $>10^{3}$ & $>10^{3}$\tabularnewline
Constr I & 52 (3.75) & 244 (10.8) & 636 (20.5) & $>10^{3}$ & $>10^{3}$\tabularnewline
Constr II & 46 (1.65) & 254 (8.18) & $>10^{3}$ & $>10^{3}$ & $>10^{3}$\tabularnewline
HSS & 87 (19.8) & 378 (72.8) & 843 (212) & $>10^{3}$ & $>10^{3}$\tabularnewline
\hline 
$6\kappa^{1/4}$ & \multicolumn{1}{c|}{11} & \multicolumn{1}{c|}{34} & \multicolumn{1}{c|}{107} & 337 & 1067\tabularnewline
ADGM & 12 (1.36) & 28 (4.01) & 116 (7.77) & 199 (23.15) & 431 (69.54)\tabularnewline
ADGM(5) & 13 (1.65) & 45 (5.98) & $>10^{3}$ & $>10^{3}$ & $>10^{3}$\tabularnewline
ADGM(10) & 12 (1.28) & 37 (4.07) & $>10^{3}$ & $>10^{3}$ & $>10^{3}$\tabularnewline
ADGM(25) & 12 (1.37) & 30 (4.61) & $>10^{3}$ & $>10^{3}$ & $>10^{3}$\tabularnewline
\hline 
\end{tabular}\hfill{}\medskip{}

\caption{\label{tab:setup}Associated max. initial set-up CPU time in seconds}

\hfill{}%
\begin{tabular}{c|c|c|c}
 & Factoring  & Factoring  & Forming and eigendecomposing\tabularnewline
$n$ & $B^{T}B=L_{B}L_{B}^{T}$ & $D=L_{D}L_{D}^{T}$ & $AD^{-1}A^{T}=V\Lambda^{-1}V^{T}$\tabularnewline
\hline 
1000 & \multicolumn{1}{c|}{0.0658} & \multicolumn{1}{c|}{0.0325} & \multicolumn{1}{c}{0.3610}\tabularnewline
3000 & 1.1122 & 0.4802 & 7.0947\tabularnewline
\hline 
\end{tabular}\hfill{}
\end{table}
We solved 1000 random problems with $n=1000$, $1\le m\le n$, $1\le\ell\le m$,
and 1000 random problems with $n=3000$, $1\le m\le n$, $1\le\ell\le m$,
using ADMM, ADMM-GMRES, and the four preconditioned Krylov methods
described above, on an Intel Core i7-3960X CPU with six 3.30 GHz cores.
All six methods were set to terminate at 1000 iterations. Accuracy
was measured as the relative residual norm with respect to the saddle-point
equation (\ref{eq:saddle}). 

Tables~\ref{tab:1k}~\&~\ref{tab:3k} show the number of iterations
and CPU time to $10^{-6}$ accuracy. Table~\ref{tab:setup} shows
the associated set-up times for the three oracles in Assumption~\ref{ass:precomp}.
ADMM and all four of the preconditioner Krylov methods converge in
$O(\sqrt{\kappa})$ iterations, but ADMM-GMRES consistently converges
in $O(\kappa^{1/4})$ iterations. This square-root factor acceleration
is large enough to offset the high per-iteration cost of the method
in every case. However, note that the preconditioner ``Constr I''
does not require access to the eigendecomposition of $AD^{-1}A^{T}$,
and so enjoys a considerable set-up time advantage over ADMM-GMRES.
Once the condition number exceeds $\kappa\ge10^{4}$, the square-root
acceleration becomes large enough to offset the fairly hefty cost
of computing the eigendecomposition, making it the fastest overall.
The restarted GMRES variant enjoys some of this acceleration, but
is also susceptible to stalling once the problem becomes sufficiently
ill-conditioned. 

The constraint preconditioner ``Constr I'' performs surprisingly
well for the $n=1000$ examples in Table~\ref{tab:1k}, consistently
outperforming its $O(\sqrt{\kappa})$ iteration bound. Examining closer,
however, we find this to be an artifact of the finite convergence
property of GMRES. Once the problem size is increased to $n=3000$,
the method is no longer able to solve ill-conditioned problems with
$\kappa\ge10^{6}$. 

In all of these examples, the per-iteration costs remain approximately
constant\textemdash even for methods that relied on GMRES\textemdash due
to the relatively high cost of the preconditioners. Profiling the
code, we find that each GMRES iteration takes no more than 4 milliseconds
to execute. By contrast, even our fastest preconditioner (Constr II)
requires $\sim36$ milliseconds per application. 

\section{\label{sec:Newton_dir}Solving the SDP Newton subproblem}

Now, we consider using ADMM-GMRES to solve the Newton subproblem associated
with an interior-point solution of the semidefinite program in (\ref{eq:SDP}).
Recall from Section~\ref{subsec:applic_SDP} that the Newton subproblem
solved at each interior-point iteration has the general form
\begin{alignat}{2}
 & \text{minimize } & \frac{1}{2}\|W^{1/2}(X-\hat{X})W^{1/2}\|_{F}^{2}+p^{T}z\label{eq:matECQPNewt}\\
 & \text{subject to } & X+\sum_{i=1}^{m}z_{i}B_{i}=Q,\nonumber 
\end{alignat}
in which all matrices are $\theta\times\theta$ real symmetric. The
symmetric positive definite matrix $W$ is known as the \emph{scaling
matrix}, and is generally fully-dense. Different interior-point methods
differ in how the scaling matrix $W$ is constructed, but in every
case, the matrix becomes progressively ill-conditioned as the interior-point
method makes progress towards the solution. To be specific, the matrix
$W$ has a condition number $\mathrm{cond}(W)=\Theta(1/\epsilon)$
at an interior-point iterate with duality gap $\epsilon$, and this
gives (\ref{eq:matECQPNewt}) the condition number of $\kappa=\Theta(1/\epsilon^{2})$;
see~\cite{wright1992interior} and also~\cite{toh2002solving,toh2004solving}.
We must solve this highly ill-conditioned problem to a similar level
of accuracy as the current duality gap $\epsilon$ in order for further
progress to be made.

Indeed, (\ref{eq:matECQPNewt}) is just an instance of (\ref{eq:ecqp}).
To see this, we define the \emph{vectorization} $\vector X$ of a
given $\theta\times\theta$ matrix $X$ as the size-$\theta^{2}$
column vector made up of sequential columns of $X$ stacked on top
of each other, and the \emph{Kronecker product} $A\otimes B$ implicitly
to satisfy the Kronecker identity $(A\otimes B)\vector X=\vector(BXA^{T})$.
Using these two operations, we can rewrite (\ref{eq:matECQPNewt})
as

\begin{alignat}{2}
 & \text{minimize } & \frac{1}{2}x^{T}(W\otimes W)x+c^{T}x+p^{T}z\label{eq:ECQPNewt}\\
 & \text{subject to } & x+Bz=q,\nonumber 
\end{alignat}
where $x=\vector X,$ $c=-\vector(W\hat{X}W),$ $B=[\vector B_{1},\ldots,\vector B_{m}]$,
and $q=\vector Q$. This is an instance of (\ref{eq:ecqp}), over
the variables $x\in\R^{n}$ of dimension $n=\theta^{2}$ and $z\in\R^{m}$
of dimension $m$, and subject to $\ell=n=\theta^{2}$ equality constraints.
Note that $B$ must have full column-rank in order for the associated
SDP to be nondegenerate~\cite{alizadeh1997complementarity}, so we
must always have $m\le n$.

Standard interior-point methods solve (\ref{eq:ECQPNewt}) directly
by forming and factoring (the Schur complement of) its KKT equations,
in cubic $O(n^{3/2}m+nm^{2}+m^{3})$ time and quadratic $\Theta(m^{2})$
memory. The goal of this section is solve (\ref{eq:ECQPNewt}) at
reduced cost using ADMM and ADMM-GMRES. In Section~\ref{subsec:Implementation}
below, we explain how each iteration of ADMM can be performed in as
low as $O(n^{3/2}+m)$ time and $O(n+m)$ memory. Then, in Sections~\ref{subsec:The-SDPLIB-problems}
and~\ref{subsec:The-DIMACS-problems}, we show that ADMM-GMRES solves
the Newton subproblems associated with an interior-point solution
of the SDPLIB~\cite{borchers1999SDPLIB} and DIMACS~\cite{pataki2002dimacs}
benchmark problems to $\epsilon$-accuracy in $O(\kappa^{1/4}\log\epsilon^{-1})$
iterations. 

\subsection{\label{subsec:Implementation}Implementation}

We begin by noting that the $n\times n$ dense matrix $W\otimes W$
can be diagonalized in $O(n^{3/2}$) time and $O(n)$ memory. This
arises from the fact that $(A\otimes B)(C\otimes D)=AC\otimes BD$.
Once we have an eigendecomposition for the $\theta\times\theta$ matrix
$W=V\Lambda V^{T}$ in $O(\theta^{3})$ time and $O(\theta^{2})$
memory, we immediately have an eigendecomposition for the $\theta^{2}\times\theta^{2}$
matrix $W\otimes W=(V\otimes V)(\Lambda\otimes\Lambda)(V\otimes V)^{T}$.
This insight gives us explicit values for the Lipschitz constant $L=\lambda_{\max}(W\otimes W)=\lambda_{\max}^{2}(W)$
and the strong convexity constant $\mu=\lambda_{\min}(W\otimes W)=\lambda_{\min}^{2}(W)$,
and shows that $\kappa=L/\mu=\mathrm{cond}(W)^{2}=\Theta(1/\epsilon^{2})$.

Now, we apply ADMM to (\ref{eq:ECQPNewt}), and obtain the following
iterations\begin{subequations}
\begin{align}
x^{k+1} & =(\beta^{-1}W\otimes W+I)^{-1}(c+q-Bz^{k}-y^{k})\label{eq:s-update}\\
z^{k+1} & =(B^{T}B)^{-1}[\beta^{-1}p-B^{T}(x^{k+1}+y^{k}-q)]\label{eq:y-update}\\
y^{k+1} & =y^{k}+(x^{k+1}+Bz^{k+1}-q).\label{eq:xt-update}
\end{align}
\end{subequations}After computing the eigendecomposition $W=V\Lambda V^{T}$,
we set the algorithm parameter $\beta$ to $\sqrt{\mu L}=\lambda_{\max}(W)\lambda_{\min}(W)$,
in order to for the sequence to converge in $O(\sqrt{\kappa})$ iterations. 

Each iteration requires a single matrix-vector product with $(\beta^{-1}W\otimes W+I)^{-1},$
$B,$ $B^{T}$, and $(B^{T}B)^{-1}$. The first matrix-vector product
$\vector X=(\beta^{-1}W\otimes W+I)^{-1}\vector Y$ can be efficiently
implemented in $O(n^{3/2})$ time and $O(n)$ memory using the eigendecomposition
$W=V\Lambda V^{T}$ and the following formula 
\[
X=V\left(\left[\frac{1}{\beta^{-1}\lambda_{i}\lambda_{j}+1}\right]_{i,j=1}^{\theta}\circ(V^{T}YV)\right)V^{T},
\]
obtained by diagonalizing $(\beta^{-1}W\otimes W+I)=(V\otimes V)(\beta^{-1}\Lambda\otimes\Lambda+I)(V\otimes V)^{T}$.
Here, $\circ$ denotes the element-wise Hadamard product. 

The cost of matrix-vector products with $B$, $B^{T}$, $(B^{T}B)^{-1}$,
however, depends on the sparsity of the SDP to be solved. For many
SDPs, particularly those that arise from combinatorial problems, the
matrix $B$ is highly sparse, and the matrix $B^{T}B$ admits a sparse
Cholesky factorization, so all three operations can be performed in
linear $O(m)$ time. However, for other problems, $B^{T}B$ may suffer
from catastrophic fill-in, and in this case, the matrix-vector product
$(B^{T}B)^{-1}$ may require up to cubic $O(m^{3})$ time and quadratic
$O(m^{2})$ memory to implement. This issue of factoring $B^{T}B$
is common to all ADMM-based approaches to SDPs; see~\cite[Rem.2]{wen2010alternating},
\cite[Sec.4]{odonoghue2016conic} and the references therein. In some
cases, an iterative method like conjugate gradients may be used~\cite[Ch.4]{boyd2011distributed},
possibly alongside an incomplete factorization preconditioner, though
this can greatly increase the per-iteration cost, thereby diminishing
the appeal of ADMM.

\subsection{\label{subsec:The-SDPLIB-problems}The SDPLIB problems}

\begin{figure}
\hfill{}\subfloat[]{\includegraphics[width=0.45\columnwidth]{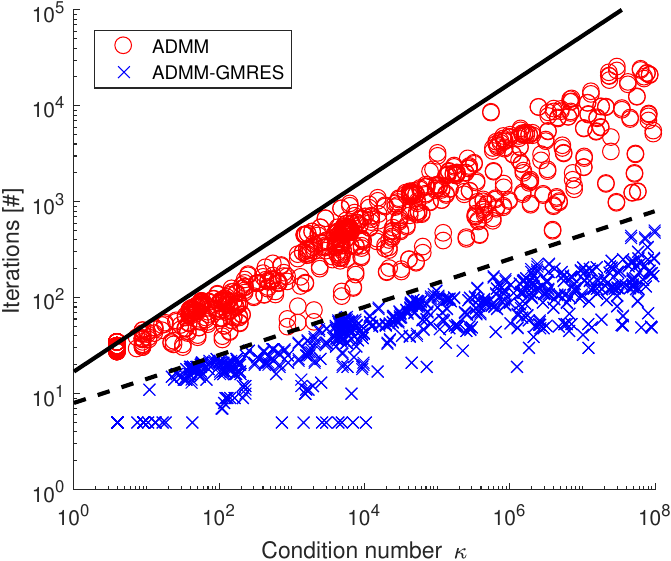}

}\hfill{}\subfloat[]{\includegraphics[width=0.45\columnwidth]{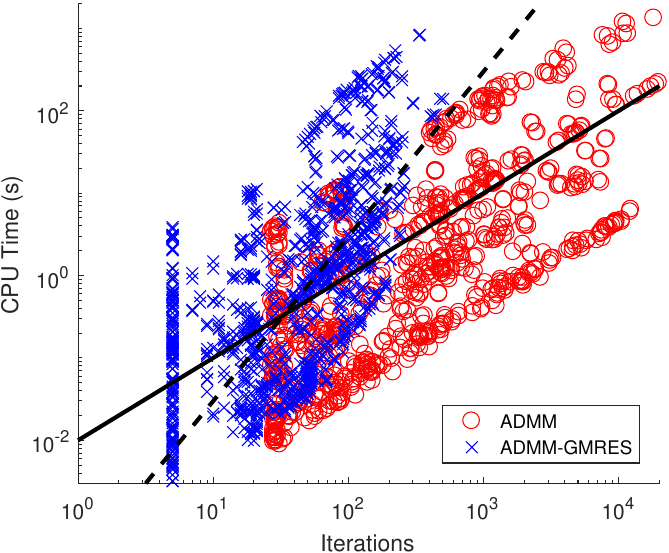}

}\hfill{}

\hfill{}\subfloat[]{\includegraphics[width=0.45\columnwidth]{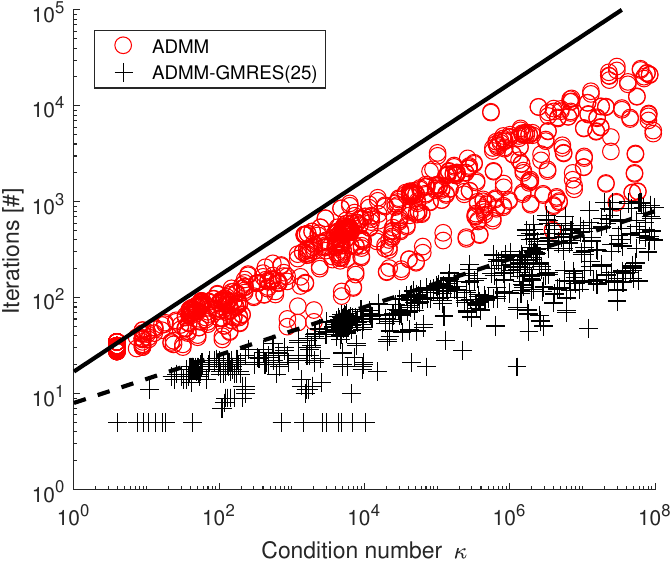}

}\hfill{}\subfloat[]{\includegraphics[width=0.45\columnwidth]{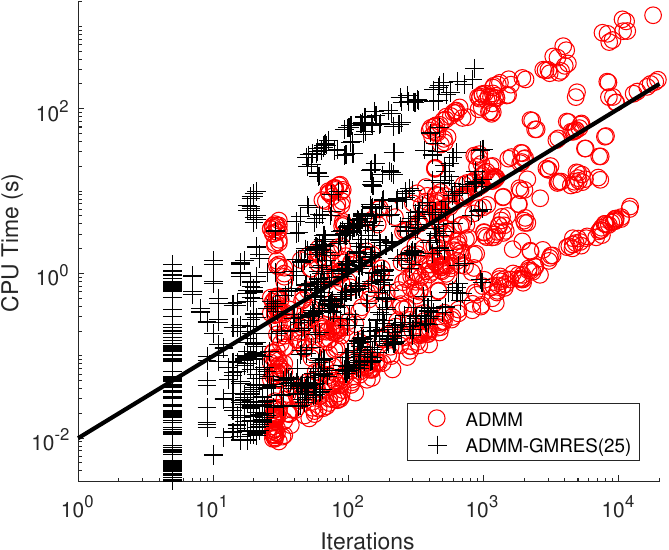}

}\hfill{}

\caption{\label{fig:it_conv_ipm}Iterations to $\epsilon=10^{-6}$ residual
convergence for the 1038 Newton direction problems described in-text:
(a) \& (b) ADMM vs ADMM-GMRES; (c) \& (d) ADMM vs ADMM-GMRES(25)}
\end{figure}
We generate instances of (\ref{eq:ECQPNewt}) using SeDuMi~\cite{sturm1999using}
over the 80 problems in the SDPLIB suite~\cite{borchers1999SDPLIB}
with $m\le700$. This collection encompasses a diversity of practical
semidefinite programs, and $m$ is small enough so that the matrix
$B^{T}B$ may always be inverted at a reasonable cost. At the same
time, the iterates $u^{k}$ have dimensions up to $n+\ell+m\le789264$
(for the problem truss8), which is large enough for the comparisons
to be realistic. For each problem, the predictor and corrector Newton
subproblems with $\kappa\le10^{8}$ are extracted and solved using
ADMM and ADMM-GMRES on an Intel Xeon E5-2687W CPU with eight 3.10
GHz cores. The stopping condition is set to be $10^{-6}$ relative
residual, i.e. when an iterate $u^{k}$ is found such that $\|u^{k}-u^{\star}\|_{M}/\|u^{\star}\|_{M}\le10^{-6}$.
The maximum number of iterations for both methods is capped at 1000. 

Fig.~\ref{fig:it_conv_ipm}a shows the number of iterations to convergence.
Results validate the $O(\sqrt{\kappa})$ figure expected of ADMM,
and the $O(\kappa^{1/4})$ figure expected of GMRES. In fact, the
multiplicative constants associated with each appear to be very similar
to the results shown earlier in Fig.~\ref{fig:first_comparison}.
Fig.~\ref{fig:it_conv_ipm}b compares the associated CPU times with
the number of iterations. The per-iteration cost of ADMM-GMRES is
constant for small $k$, but grows linearly with $k$ beyond about
30 iterations. For many of the problems considered, the square-root
factor reduction in iterations to convergence is offset by the quadratic
growth in computation time, and both methods end up using a similar
amount of time, despite the considerable difference in iteration count. 

A practical implementation of ADMM-GMRES will require the use of a
limited-memory version of GMRES. We consider the simplest approach
of restarting every 25 iterations; the results are shown in Figs.~\ref{fig:it_conv_ipm}c~\&~\ref{fig:it_conv_ipm}d.
The restarted variant requires a factor of two more iterations to
converge when compared to the usual algorithm. The amortized per-iteration
cost of restarted ADMM-GMRES is a factor of two times higher that
of basic ADMM, but the method also converges in significantly fewer
iterations.

\subsection{\label{subsec:The-DIMACS-problems}The DIMACS problems}

\begin{table}
\caption{\label{tab:dimacs}Solving DIMACS Problems using SeDuMi modified to
compute search directions using ADMM-GMRES(25). ``CPU'' is total
CPU time in seconds, ``Iter'' is the total inner ADMM-GMRES iterations
taken, and ``feas'' and ``opt'' are defined in-text.}

\hfill{}%
\begin{tabular}{c|c|c|c|c|c|c|c|c}
Key & Name & $m$ & $\theta$ & $N$ & CPU & Iter & feas & opt\tabularnewline
\hline 
a & hamming\_7\_5\_6 & 1793 & 128 & $3.46\times10^{4}$ & 2.86 & 300 & 6.59 & 7.92\tabularnewline
\hline 
b & hamming\_8\_3\_4 & 16129 & 256 & $1.47\times10^{5}$ & 11.5 & 347 & 6.37 & 8.16\tabularnewline
\hline 
c & hamming\_9\_5\_6 & 53761 & 512 & $5.78\times10^{5}$ & 85.2 & 422 & 6.45 & 7.89\tabularnewline
\hline 
d & hamming\_9\_8 & 2305 & 512 & $5.27\times10^{5}$ & 70.4 & 411 & 5.39 & 7.47\tabularnewline
\hline 
e & hamming\_10\_2 & 23041 & 1024 & $2.12\times10^{6}$ & 256 & 328 & 5.40 & 7.94\tabularnewline
\hline 
f & hamming\_11\_2 & 56321 & 2048 & $8.44\times10^{6}$ & 1262 & 217 & 3.98 & 6.36\tabularnewline
\hline 
g & toruspm3-8-50 & 512 & 512 & $5.25\times10^{5}$ & 201 & 1651 & 2.53 & 4.91\tabularnewline
\hline 
h & torusg3-8 & 512 & 512 & $5.25\times10^{5}$ & 950 & 8906 & 3.64 & 6.07\tabularnewline
\hline 
i & torusg3-15 & 3375 & 3375 & $2.28\times10^{7}$ & 26721 & 3280 & 1.68 & 4.88\tabularnewline
\hline 
j & toruspm3-15-50 & 3375 & 3375 & $2.28\times10^{7}$ & 36782 & 4387 & 2.20 & 5.42\tabularnewline
\hline 
\end{tabular}\hfill{}
\end{table}
\begin{figure}
\hfill{}\subfloat[]{\includegraphics[width=0.49\columnwidth]{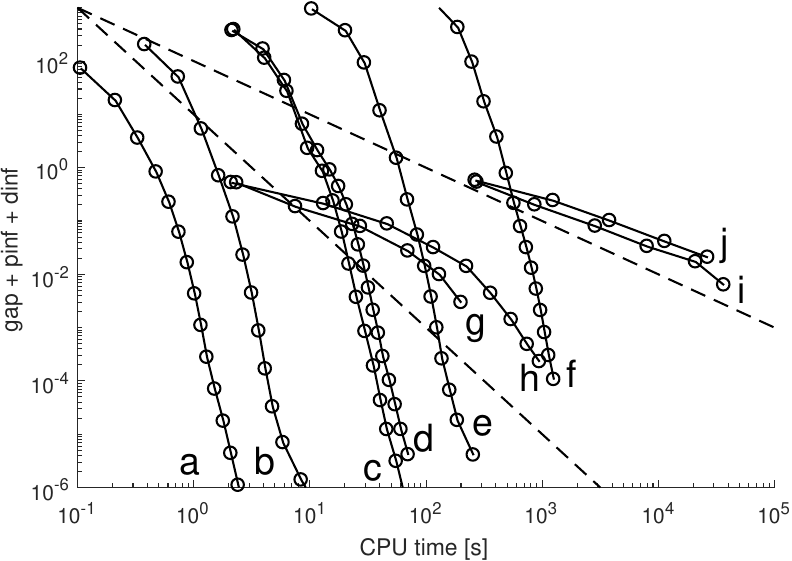}

}\hfill{}\subfloat[]{\includegraphics[width=0.49\columnwidth]{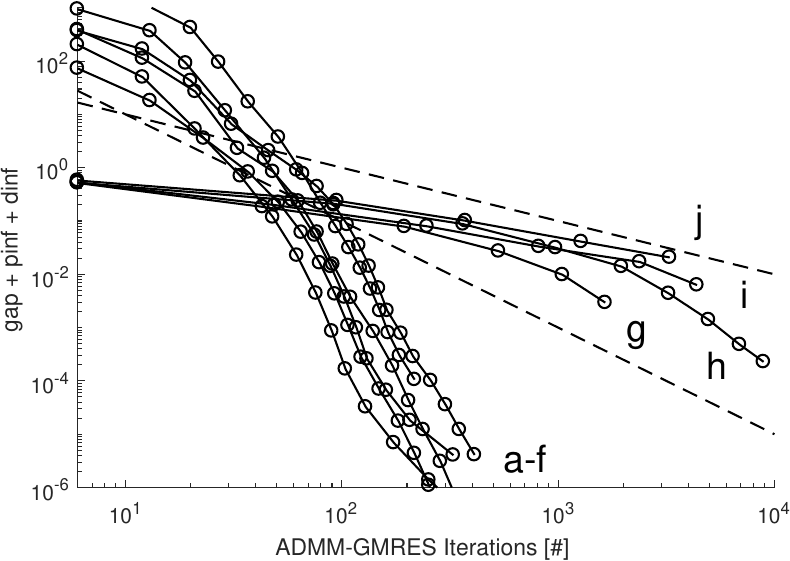}

}\hfill{}

\caption{\label{fig:dimacs}The convergence behavior of embedding ADMM-GMRES
within SeDuMi for the DIMACS problems: (a) against CPU time; (b) against
inner ADMM-GMRES iterations. Data for the keys are shown in Table~\ref{tab:dimacs}.
Each marker shows a single interior-point iteration. The dash lines
indicate sublinear $O(1/k)$ and $O(1/k^{2})$ error rates.}
\end{figure}
Finally, we incorporate ADMM-GMRES within SeDuMi, as a set of inner
iterations within an outer interior-point method, in order to solve
large-scale SDPs from the Seventh DIMACS Implementation Challenge~\cite{pataki2002dimacs}.
In other words, we modify SeDuMi to use ADMM-GMRES to compute the
Newton search directions, in lieu of its internal Cholesky-based solver,
while leaving the remainder of the solver unchanged. The Newton subproblems
are large enough to prevent the full GMRES from being used, so we
restart GMRES every 25 iterations, allowing up to 1000 inner iterations
to be performed per outer interior-point iteration. The inner ADMM-GMRES
iterations are terminated when an iterate $u^{k}$ is found such that
$\|H(u^{k}-u^{\star})\|_{\infty}\le\tau$, where $H$ is the KKT matrix
for (\ref{eq:ECQPNewt}), and $\tau$ is an absolute tolerance as
specified by the SeDuMi algorithm. Typically, $\tau$ is around one
order of magnitude smaller than the current duality gap $\epsilon$.
The outer interior-point iterations are terminated either by converging
to the desired solution accuracy, or prematurely if SeDuMi considers
the computed search direction to be too inaccurate to make further
progress. 

The numerical experiments are performed on an Intel Xeon E5-2609 v4
CPU with eight 1.70 GHz cores. Table~\ref{tab:dimacs} shows the
results, and Fig.~\ref{fig:dimacs} plots the progress of the interior-point
iterates over different interior-point iterations against time and
against inner ADMM-GMRES iterations. Here, $N=n+\ell+m=2\theta^{2}+m$
refers to the total number of primal-dual variables in the corresponding
(\ref{eq:ecqp}) problem. The accuracy of each interior-point iterate
$\{X,y,S\}$ is quantified by the number of decimal \emph{digits of
feasibility $\mathrm{feas}=-\log_{10}(\mathrm{pinf}+\mathrm{dinf})$}
and \emph{digits of optimality} $\mathrm{opt}=-\log_{10}(\mathrm{gap}),$
which are themselves defined in terms of the dimensionless DIMACS
metrics~\cite{pataki2002dimacs}:
\begin{gather*}
\mathrm{pinf}=\frac{\sqrt{\sum_{i}(\tr B_{i}Y-p_{i})^{2}}}{1+\|p\|_{2}},\qquad\mathrm{dinf}=\frac{\|\sum z_{i}B_{i}+X-Q\|_{F}}{1+\|Q\|_{F}}.\\
\mathrm{gap}=\frac{|\tr QY-p^{T}z|}{1+|\tr QY|+|p^{T}z|}.
\end{gather*}

The results show that ADMM-GMRES with restarts is able to converge
within $O(\kappa^{1/4})$ iterations for 8 out of the 10 problems,
namely the problems labeled from ``a'' to ``h''. Their corresponding
convergence curves demonstrate a time complexity of $O(1/\sqrt{\epsilon})$
and an error rate of $O(1/k^{2})$ at the $k$-th iteration. This
is the ``accelerated'' rate that we described earlier in Section~\ref{subsec:applic_SDP},
typically obtained by ``fast'' first-order methods. For the remaining
2 problems, however, the method only converges in $O(\sqrt{\kappa})$
iterations, with a time complexity of $O(1/\epsilon)$ and an error
rate of $O(1/k)$ at the $k$-th iteration. This is the usual rate
attained by solving the interior-point Newton subproblem using a standard
iterative method like conjugate gradients.

\section{Conclusions and future work}

In this paper, we have provided theoretical and numerical evidence
that ADMM-GMRES can consistently converge in $O(\kappa^{1/4})$ iterations
for a smooth strongly convex quadratic objective, despite a worst-case
bound of $O(\sqrt{\kappa})$ iterations. The order-of-magnitude reduction
in iterations over the basic ADMM method was widely observed for both
randomized examples and in the Newton subproblems for the interior-point
solution of semidefinite programs. These results confirm the possibility
for an over-relaxation scheme, momentum scheme, or otherwise, to significantly
accelerate the convergence of ADMM, beyond the constant factor typically
observed for existing schemes, and suggest the direct use of ADMM-GMRES
as a practical solution method.

It remains an open question whether the same sort of acceleration
can be extended to ADMM for general nonquadratic objectives. One possible
approach is to use GMRES to a linearized approximation of the nonlinear
fixed-point equation, in a Krylov-Newton method~\cite{brown1994convergence}.
Alternatively, a Broyden-like secant approximation may be constructed
from previous iterates, and used to extrapolate the current step,
in an Anderson acceleration method~\cite{walker2011anderson}. Both
approaches reduce to ADMM-GMRES in the case of quadratic objectives,
but further work is needed to understand their effectiveness.

\section*{Acknowledgments}

We wish to thank Jos\'{e} E. Serrall\'{e}s for proofreading an early
draft, and for assisting with the numerical results; L\'{a}szl\'{o}
Mikl\'{o}s Lov\'{a}sz for discussions on random matrix theory that
led to Section~\ref{subsec:explain_empirical}. A large part of the
paper was written during R.Y. Zhang's visit to UC Berkeley as postdoctoral
scholar, and he would like to thank his faculty mentor Javad Lavaei
for his warm accommodation. 

\bibliographystyle{siam}
\bibliography{refs}

\end{document}